\documentclass[12pt]{article}
\usepackage{amssymb}
\usepackage{mathrsfs}
\usepackage[hypertex]{hyperref}
\usepackage{amsfonts}
\usepackage{graphicx}
\usepackage{latexsym,amsmath,amssymb,amsfonts,amsthm}

\newcommand{\bcen}{\begin{center}}
\newcommand{\ecen}{\end{center}}
\newtheorem{theorem}{Theorem}[section]
\newtheorem{lemma}[theorem]{Lemma}

\newtheorem{corollary}[theorem]{Corollary}

\newtheorem{remark}[theorem]{Remark}

\begin{document}
\setcounter{page}{1}
\title{On the Ashbaugh-Benguria type conjecture about lower-order Neumann eigenvalues of the Witten-Laplacian}
\author{Ruifeng Chen,~~ Jing Mao$^{\ast}$}

\date{}
\protect \footnotetext{\!\!\!\!\!\!\!\!\!\!\!\!{~$\ast$ Corresponding author} }
\maketitle ~~~\\[-15mm]

\begin{center}
{\footnotesize Faculty of Mathematics and Statistics,\\
 Key Laboratory of Applied
Mathematics of Hubei Province, \\
Hubei University, Wuhan 430062, China\\
Emails: gchenruifeng@163.com (R. F. Chen), jiner120@163.com (J. Mao)
 }
\end{center}


\begin{abstract}
An isoperimetric inequality for lower order nonzero Neumann
eigenvalues of the Witten-Laplacian on bounded domains in a
Euclidean space or a hyperbolic space has been proven in this paper.
About this conclusion, we would like to point out two things:
\begin{itemize}

\item It strengthens the well-known
Szeg\H{o}-Weinberger inequality for nonzero Neumann eigenvalues of
the classical free membrane problem given in [J. Rational Mech.
Anal. {\bf 3} (1954) 343--356] and [J. Rational Mech. Anal. {\bf 5}
(1956) 633--636];

\item Recently, Xia-Wang [Math. Ann. {\bf 385} (2023)
863--879] gave a very important progress to the celebrated
conjecture of M. S. Ashbaugh and R. D. Benguria proposed in [SIAM J.
Math. Anal. {\bf 24} (1993) 557--570]. It is easy to see that our
conclusion here covers Xia-Wang's this progress as a special case.
\end{itemize}
In this paper, we have also proposed two open problems which can be
seen as a generalization of Ashbaugh-Benguria's conjecture mentioned
above.
 \end{abstract}


\section{Introduction}
\renewcommand{\thesection}{\arabic{section}}
\renewcommand{\theequation}{\thesection.\arabic{equation}}
\setcounter{equation}{0}

Let $(M^{n},\langle\cdot,\cdot\rangle)$ be an $n$-dimensional
($n\geq2$) complete Riemannian manifold with the metric
$g:=\langle\cdot,\cdot\rangle$. Let $\Omega\subseteq M^{n}$ be a
 domain in $M^n$, and $\phi\in C^{\infty}(M^n)$ be a
smooth\footnote{~In fact, one might see that $\phi\in C^{2}$ is
suitable to derive our main conclusions in this paper. However, in
order to avoid a little bit boring discussion on the regularity of
$\phi$ and following the assumption on conformal factor $e^{-\phi}$
for the notion of \emph{smooth metric measure spaces} in many
literatures (including of course those cited in this paper), without
specification, we wish to assume that $\phi$ is smooth on the domain
$\Omega$. } real-valued function defined on $M^n$. In this setting,
on $\Omega$, the following elliptic operator
\begin{eqnarray*}
\Delta_{\phi}:=\Delta-\langle\nabla\phi,\nabla\cdot\rangle
\end{eqnarray*}
can be well-defined, where $\nabla$, $\Delta$ are the gradient and
the Laplace operators on $M^{n}$, respectively. The operator
$\Delta_{\phi}$ w.r.t. the metric $g$ is called the
\emph{Witten-Laplacian} (also called the \emph{drifting Laplacian}
or the \emph{weighted Laplacian}). The $K$-dimensional
Bakry-\'{E}mery Ricci curvature $\mathrm{Ric}^{K}_{\phi}$ on $M^{n}$
can be defined as follows
 \begin{eqnarray*}
\mathrm{Ric}^{K}_{\phi}:=\mathrm{Ric}+\mathrm{Hess}\phi-\frac{d\phi\otimes
d\phi}{K-n-1},
 \end{eqnarray*}
where $\mathrm{Ric}$ denotes the Ricci curvature tensor on $M^{n}$,
and $\mathrm{Hess}$ is the Hessian operator on $M^{n}$ associated to
the metric $g$. Here $K>n+1$ or $K=n+1$ if $\phi$ is a constant
function. When $K=\infty$, the so-called $\infty$-dimensional
Bakry-\'{E}mery Ricci curvature $\mathrm{Ric}_{\phi}$ (simply,
\emph{Bakry-\'{E}mery Ricci curvature} or \emph{weighted Ricci
curvature}) can be defined as follows
\begin{eqnarray*}
 \mathrm{Ric}_{\phi}:=\mathrm{Ric}+\mathrm{Hess}\phi.
\end{eqnarray*}
These notions were introduced by D. Bakry and M. \'{E}mery in
\cite{BE}. Many interesting results (under suitable assumptions on
the Bakry-\'{E}mery Ricci curvature) have been obtained, and we wish
to mention briefly several ones:
\begin{itemize}

\item For Riemannian manifolds\footnote{~Without
specifications, generally, in this paper same symbols have the same
meanings.} $(M^{n},g)$ endowed with a weighted measure
$d\eta:=e^{-\phi}dv$, where $dv$ denotes the Riemannian volume
element (or Riemannian density) w.r.t. the metric $g$, Wei and Wylie
\cite{WW} proved mean curvature and volume comparison results when
the Bakry-\'{E}mery Ricci curvature $\mathrm{Ric}_{\phi}$ is bounded
from below and $\phi$ or $|\nabla\phi|$ is bounded, improving the
classical ones (i.e., when $\phi$ is constant). As described by J.
Mao (the corresponding author here) in \cite[pp. 31-32]{JM1}, one
might have an illusion that smooth metric measure spaces are not
necessary to be studied since they are simply obtained from
corresponding Riemannian manifolds by adding a conformal factor to
the Riemannian measure. However, they do have many differences. For
instance, when $\mathrm{Ric}_{\phi}$ is bounded from below, the
Myer's theorem, Bishop-Gromov's volume comparison, Cheeger-Gromoll's
splitting theorem and Abresch-Gromoll's excess estimate cannot hold
as in the Riemannian case. Moreover, in order to let readers have a
deep impression and a nice comprehension on those differences, Mao
\cite[page 32]{JM1} has also repeated briefly an interesting example
(given in \cite[Example 2.1]{WW}) to make an explanation therin.
More precisely, for the metric measure space
$(\mathbb{R}^{n},g_{\mathbb{R}^{n}},e^{-\phi}dv_{\mathbb{R}^{n}})$,
where $g_{\mathbb{R}^{n}}$ is the usual Euclidean metric of the
Euclidean $n$-space $\mathbb{R}^n$, and $dv_{\mathbb{R}^{n}}$
denotes the Euclidean volume density related to
$g_{\mathbb{R}^{n}}$, if $\phi(x)=\frac{\lambda}{2}|x|^2$ for
$x\in\mathbb{R}^n$, then $\mathrm{Hess}=\lambda g_{\mathbb{R}^n}$
and $\mathrm{Ric}_{\phi}=\lambda g_{\mathbb{R}^n}$.  Therefore, from
this example we know that unlike in the case of Ricci curvature
bounded from below uniformly by some positive constant, a metric
measure space is not necessarily compact provided
$\mathrm{Ric}_{\phi}\geq\lambda$ and $\lambda>0$. Hence, it is
meaningful to study geometric problems in smooth metric measure
spaces.

\item Perelman's $\mathcal{W}$-entropy formula for the heat
equation associated with the Witten Laplacian on complete Riemannian
manifolds via the Bakry-\'{E}mery Ricci curvature tensor has been
investigated by Li \cite{XDL}. In fact, under the assumption that
the $m$-dimensional Bakry-\'{E}mery Ricci curvature is bounded from
below, Li \cite[Theorem 2.3]{XDL} obtained an analogue of Perelman's
entropy formula for the $\mathcal{W}$-entropy of the heat kernel of
the Witten Laplacian on complete Riemannian manifolds with some
natural geometric conditions. In particular, by this fact, he proved
a monotonicity theorem and a rigidity theorem for the
$\mathcal{W}$-entropy on complete Riemannian manifolds with
\emph{nonnegative} $m$-dimensional Bakry-\'{E}mery Ricci curvature.

\item Mao and his collaborators \cite[Theorems 4.1 and 4.4, Corollary 4.2]{DMWW} investigated the
buckling problem of the drifting Laplacian, and \textbf{firstly}
obtained some universal inequalities for eigenvalues of the same
problem on bounded connected domains in the Gaussian shrinking
solitons
$$\left(\mathbb{R}^{n},g_{\mathbb{R}^{n}},e^{-\frac{|x|^2}{4}}dv_{\mathbb{R}^{n}},\frac{1}{2}\right)$$
 and
some general product solitons of the type
$$\left(\Sigma\times\mathbb{R},g,e^{-\frac{\kappa t^2}{2}}dv,\kappa\right),$$
where $\Sigma$ is an Einstein manifold with constant Ricci curvature
$\kappa$, and $t\in\mathbb{R}$ is the parameter defined along the
line $\{x\}\times\mathbb{R}$, $x\in\Sigma$. Besides, as interpreted
in \cite[Remark 4.3]{DMWW}, for a self-shrinker, if the weighted
function $\phi$ was chosen to be $\phi=\frac{|x|^2}{4}$, then the
drifting Laplacian considered in \cite{DMWW} degenerates into the
operator $\mathcal{L}:=\Delta-\frac{1}{2}\langle
x,\nabla(\cdot)\rangle$ which was introduced by Colding-Minicozzi II
\cite{CMII} to study self-shrinker hypersurfaces. For the Dirichlet
eigenvalue problem of the operator $\mathcal{L}$, Cheng-Peng
\cite{CP} have obtained some universal inequalities. From this
viewpoint, \cite[Theorem 4.1 and Corollary 4.2]{DMWW} can be
regarded as conclusions for the buckling problem of the operator
$\mathcal{L}$.

\end{itemize}
Except \cite{DMWW,JM1}, Mao also has some other interesting works
related to the Witten-Laplacian -- see, e.g.,
\cite{DM1,LMWZ,JM2,MTZ,YWMD}.

Using the conformal measure $d\eta=e^{-\phi}dv$, the notion,
\emph{smooth metric measure space} $(M^{n},g,d\eta)$, can be
well-defined, which is actually the given Riemannian manifold
$(M^{n},g)$ equipped with the weighted measure $d\eta$. Smooth
metric measure space $(M^{n},g,d\eta)$ sometimes is also called the
\emph{weighted measure space}. For the smooth metric measure space
$(M^{n},g,d\eta)$, one can define a notion, weighted volume (or
$\phi$-volume), as follows:
\begin{eqnarray*}
|M^{n}|_{\phi}:=\int_{M^{n}}d\eta=\int_{M^n}e^{-\phi}dv.
\end{eqnarray*}
On a compact smooth metric measure space
$(\Omega,\langle\cdot,\cdot\rangle,d\eta)$, one can naturally
consider the Neumann eigenvalue problem of the Witten-Laplacian
$\Delta_{\phi}$ as follows
\begin{eqnarray} \label{eigen-1}
\left\{
\begin{array}{ll}
\Delta_{\phi} u+\mu u=0\qquad & \mathrm{in}~\Omega\subset M^{n}, \\[0.5mm]
\frac{\partial u}{\partial\vec{\nu}}=0 \qquad &
\mathrm{on}~{\partial\Omega},
\end{array}
\right.
\end{eqnarray}
where $\vec{\nu}$ is the unit normal vector along the
smooth\footnote{~The smoothness assumption for the regularity of the
boundary $\partial\Omega$ is strong to consider the eigenvalue
problem (\ref{eigen-1}) of the Witten-Laplacian. In fact, a weaker
regularity assumption that $\partial\Omega$ is Lipschitz continuous
can also assure the validity about the description of the discrete
spectrum of the eigenvalue problem (\ref{eigen-1}). However, the
Lipschitz continuous assumption might not be enough to consider some
other geometric problems involved Neumann eigenvalues of
(\ref{eigen-1}). Therefore, to avoid excessive focus on the
regularity of the boundary $\partial\Omega$ -- which is not central
to the topic of this paper -- we assume, unless otherwise stated,
that $\partial\Omega$ is smooth. This setting leads to the situation
that certain conclusions of this paper may still hold even under a
weaker regularity assumption for the boundary $\partial\Omega$,
readers who are interested in this situation could try to seek the
weakest regularity. } boundary $\partial\Omega$. The eigenvalue
problem (\ref{eigen-1}) can also be called the \emph{free membrane
problem} of the operator $\Delta_{\phi}$. It is not hard to check
that the operator $\Delta_{\phi}$ in (\ref{eigen-1}) is
\textbf{self-adjoint} w.r.t. the following inner product
\begin{eqnarray*}
(f_{1},f_{2})_{\phi}:=\int_{\Omega}f_{1}f_{2}d\eta=\int_{\Omega}f_{1}f_{2}e^{-\phi}dv,
\end{eqnarray*}
with $f_{1},f_{2}\in W^{1,2}_{\phi}(\Omega)$,  where $
W^{1,2}_{\phi}(\Omega)$ stands for the Sobolev space w.r.t. the
weighted measure $d\eta$, i.e. the completion of the set of smooth
functions $C^{\infty}(\Omega)$ under the following Sobolev norm
\begin{eqnarray*}
\|f\|_{1,2}^{\phi}:=\left(\int_{\Omega}f^{2}d\eta+\int_{\Omega}|\nabla
f|^{2}d\eta\right)^{1/2}.
\end{eqnarray*}
Then using similar arguments to those of the classical free membrane
problem of the Laplacian (i.e., the discussions about the existence
of discrete spectrum, Rayleigh's theorem, Max-min theorem, etc.
These discussions are standard, and for details, please see for
instance \cite{IC}), it is not hard to know:
\begin{itemize}
\item The self-adjoint elliptic operator $-\Delta_{\phi}$ in
(\ref{eigen-1}) \emph{only} has discrete spectrum, and all the
elements (i.e., eigenvalues) in this discrete spectrum can be listed
non-decreasingly as follows
\begin{eqnarray} \label{sequence-1}
0=\mu_{0,\phi}(\Omega)<\mu_{1,\phi}(\Omega)\leq\mu_{2,\phi}(\Omega)\leq\cdots\uparrow+\infty.
\end{eqnarray}
For each eigenvalue $\mu_{i,\phi}(\Omega)$, $i=0,1,2,\cdots$, all
the possible nontrivial functions $u$ satisfying (\ref{eigen-1}) are
called eigenfunctions belonging to $\mu_{i,\phi}(\Omega)$. Since the
first equation in  (\ref{eigen-1}) is linear, the space of
$\mu_{i,\phi}(\Omega)$'s eigenfunctions should be a vector space.
This vector space of $\mu_{i,\phi}(\Omega)$ is called eigenspace.
Each eigenspace has finite dimension, and usually the dimension of
each eigenspace is called multiplicity of the eigenvalue. It is easy
to know that eigenfunctions of the first eigenvalue
$\mu_{0,\phi}(\Omega)=0$ are nonzero constant functions, and
correspondingly, the eigenspace of $\mu_{0,\phi}(\Omega)=0$ has
dimension $1$. Eigenvalues in the sequence (\ref{sequence-1}) are
repeated according to its multiplicity. By applying the standard
variational principles, one can obtain that the $k$-th nonzero
Neumann eigenvalue $\mu_{k,\phi}(\Omega)$ can be characterized as
follows
 \begin{eqnarray}  \label{chr-1}
 \mu_{k,\phi}(\Omega)=\inf\left\{\frac{\int_{\Omega}|\nabla f|^{2}e^{-\phi}dv}{\int_{\Omega}f^{2}e^{-\phi}dv}
 \Bigg{|}f\in W^{1,2}_{\phi}(\Omega),f\neq0,\int_{\Omega}ff_{i}e^{-\phi}dv=0\right\},
 \end{eqnarray}
where $f_{i}$, $i=0,1,2,\cdots,k-1$, denotes an eigenfunction of
$\mu_{i,\phi}(\Omega)$. Specially, the first nonzero Neumann
eigenvalue $\mu_{1,\phi}(\Omega)$ of the eigenvalue problem
(\ref{eigen-1}) satisfies
\begin{eqnarray}  \label{chr-2}
 \mu_{1,\phi}(\Omega)=\inf\left\{\frac{\int_{\Omega}|\nabla f|^{2}d\eta}{\int_{\Omega}f^{2}d\eta}
 \Bigg{|}f\in W^{1,2}_{\phi}(\Omega),f\neq0,\int_{\Omega}fd\eta=0\right\}.
 \end{eqnarray}
For convenience and without confusion, in the sequel, except
specification we will write $\mu_{i,\phi}(\Omega)$ as $\mu_{i,\phi}$
directly. This convention would be also used when we meet with other
possible eigenvalue problems.
\end{itemize}
 In this paper, we focus on the Neumann eigenvalue problem
(\ref{eigen-1}) of the Witten-Laplacian and can prove an
isoperimetric inequality for the sums of the reciprocals of the
first $(n-1)$ nonzero Neumann eigenvalues of the Witten-Laplacian on
bounded domains in $\mathbb{R}^n$ or a hyperbolic space. However, in
order to state our conclusions clearly, we need to impose an
assumption on the function $\phi$ as follows:
\begin{itemize}
\item (\textbf{Property I}) $\phi$ is a function
of the Riemannian distance parameter $t:=d(o,\cdot)$ for some point
$o\in\mathrm{hull}(\Omega)$, and $\phi$ is also a non-increasing
convex function defined on $[0,\infty)$.
\end{itemize}
Here $\mathrm{hull}(\Omega)$ stands for the convex hull of the
domain $\Omega$. Clearly, if a given open Riemannian $n$-manifold
$(M^{n},g)$ was endowed with the weighted density $e^{-\phi}dv$ with
$\phi$ satisfying \textbf{Property I}, then $\phi$ would be a
\emph{\textbf{radial}} function defined on $M^{n}$ w.r.t. the radial
distance $t$, $t\in[0,\infty)$. Especially, when the given open
$n$-manifold is chosen to be $\mathbb{R}^{n}$ or $\mathbb{H}^{n}$
(i.e., the $n$-dimensional hyperbolic space of sectional curvature
$-1$), we additionally require that $o$ is the origin  of
$\mathbb{R}^{n}$ or $\mathbb{H}^{n}$.

\begin{theorem} \label{theo-1}
 Assume that the function $\phi$ satisfies \textbf{Property I}. Let $\Omega$ be a bounded domain with smooth boundary in
 $\mathbb{R}^n$, and let $B_{R}(o)$ be a ball of radius $R$ and centered at the
 origin $o$
 of
 $\mathbb{R}^{n}$ such that $|\Omega|_{\phi}=|B_{R}(o)|_{\phi}$,
 i.e. $\int_{\Omega}d\eta=\int_{B_{R}(o)}d\eta$.  Then
  \begin{eqnarray} \label{II-1}
\frac{1}{\mu_{1,\phi}(\Omega)}+\frac{1}{\mu_{2,\phi}(\Omega)}+\cdots+\frac{1}{\mu_{n-1,\phi}(\Omega)}\geq\frac{n-1}{\mu_{1,\phi}(B_{R}(o))}.
\end{eqnarray}
The equality case holds if and only if $\Omega$ is the ball
$B_{R}(o)$.
\end{theorem}

By applying the sequence (\ref{sequence-1}), i.e. the monotonicity
of Neumann eigenvalues of the Witten-Laplacian, from (\ref{II-1})
one has:

\begin{corollary} \label{coro-1}
Under the assumptions of Theorem \ref{theo-1}, we have
\begin{eqnarray} \label{II-2}
\mu_{1,\phi}(\Omega)\leq\mu_{1,\phi}(B_{R}(o)),
\end{eqnarray}
with equality holding if and only if $\Omega$ is the ball
$B_{R}(o)$. That is to say, among all bounded domains in
$\mathbb{R}^n$ having the same weighted volume, the ball $B_{R}(o)$
maximizes the first nonzero Neumann eigenvalue of the
Witten-Laplacian, provided the function $\phi$ satisfies
\textbf{Property I}.
\end{corollary}

\begin{remark}
\rm{ The spectral isoperimetric inequality (\ref{II-2}) in Corollary
\ref{coro-1} has already been proven by the authors in \cite{CM} by
suitably constructing the trial function. However, we still wish to
list it here to show the close relation between (\ref{II-1}) and
(\ref{II-2}), and show of course the significance of the spectral
isoperimetric inequality (\ref{II-1}) as well.  }
\end{remark}

\begin{remark} \label{remark1-4}
\rm{ (1) Topologically, the Euclidean $n$-space $\mathbb{R}^n$ is
two-points homogenous, so generally it seems like there is no need
to point out the information of the center for the ball $B_{R}(o)$
when describing the isometry conclusion in Theorem \ref{theo-1}.
However, for the eigenvalue problem (\ref{eigen-1}), by
(\ref{chr-1}) one knows that even on Euclidean balls, the Neumann
eigenvalues $\mu_{i,\phi}$ also depend on the weighted function
$\phi$ (except the situation that $\phi$ is a constant function).
This implies that for Euclidean balls with the same radius but
different centers, they might have different Neumann eigenvalues
$\mu_{i,\phi}$ since generally the radial function $\phi$ here has
different distributions on different balls. Therefore, we need to
give the information of the center for the ball $B_{R}(o)$ when we
investigate the possible rigidity for the equality case of
(\ref{II-1}). \\
(2) A slightly sharper version of (\ref{II-1}) has also been
obtained -- for details, see Theorem \ref{theo-3} in Section
\ref{S3} below. \\
(3) As we know, if $\phi=const.$ is a constant function,
 then the Witten-Laplacian $\Delta_{\phi}$ degenerates into the
 Laplacian $\Delta$, and correspondingly the eigenvalue problem
 (\ref{eigen-1}) becomes the classical free membrane problem of the Laplacian
 $\Delta$ as follows
\begin{eqnarray} \label{eigen-2}
\left\{
\begin{array}{ll}
\Delta u+\mu u=0\qquad & \mathrm{in}~\Omega\subset M^{n}, \\[0.5mm]
\frac{\partial u}{\partial\vec{\nu}}=0 \qquad &
\mathrm{on}~{\partial\Omega}.
\end{array}
\right.
\end{eqnarray}
Clearly, the Laplacian $-\Delta$ in (\ref{eigen-2}) only has the
discrete spectrum, and all the eigenvalues in this discrete spectrum
can be listed non-decreasingly as follows
\begin{eqnarray*}
0=\mu_{0}(\Omega)<\mu_{1}(\Omega)\leq\mu_{2}(\Omega)\leq\cdots\uparrow+\infty.
\end{eqnarray*}
The corresponding Neumann eigenvalues $\mu_{k}$ can be characterized
similarly as (\ref{chr-1})-(\ref{chr-2}) with $\phi=const.$ instead.
Now, we would like to recall some results on isoperimetric
inequalities of Neumann eigenvalues of the eigenvalue problem
(\ref{eigen-2}). For simply connected bounded domains
$\Omega\subset\mathbb{R}^2$, by using the conformal mapping
techniques, Szeg\H{o} \cite{GS} obtained
 \begin{eqnarray}  \label{II-3}
 \mu_{1}(\Omega)A(\Omega)\leq\mu_{1}(\mathbb{D})A(\mathbb{D})=\pi
 p^{2}_{1,1},
 \end{eqnarray}
 where $\mathbb{D}$ stands for a disk in the plane $\mathbb{R}^2$,
 and $A(\cdot)$ denotes the area of a given geometric object. Later,
 this result was improved by Weinberger \cite{HFW} to the higher
 dimensional case, that is, for bounded domains
 $\Omega\subset\mathbb{R}^n$, $n\geq2$, he proved
  \begin{eqnarray}  \label{II-4}
  \mu_{1}(\Omega)\leq\left(\frac{w_n}{|\Omega|}\right)^{2/n}p^{2}_{n/2,1},
  \end{eqnarray}
where $w_{n}$, $|\Omega|$ denote\footnote{~Similarly, without
confusion, in the sequel $|\cdot|$ would denote the volume of a
given geometric object.} the volume of the unit ball in
$\mathbb{R}^{n}$ and the volume of $\Omega$, respectively. Here
$p_{v,k}$ in (\ref{II-3})-(\ref{II-4}) stands for the $k$-th
positive zero of the derivative of $x^{1-v}J_{v}(x)$, with
$J_{v}(x)$ the Bessel function of the first kind of order $v$. The
equality case in (\ref{II-3}) (or (\ref{II-4})) holds if and only if
$\Omega$ is a disk (or a ball in $\mathbb{R}^n$). Clearly, from the
Szeg\H{o}-Weinberger's isoperimetric inequality (\ref{II-4}), one
knows:
\begin{itemize}
\item (\textbf{Fact A}) Among all bounded domains in
$\mathbb{R}^n$ having the same volume, the ball maximizes the first
nonzero Neumann eigenvalue of the Laplacian.
\end{itemize}
It is not hard to see that \textbf{Fact A} was covered by Corollary
\ref{coro-1} as a special case (corresponding to $\phi=const.$).
Szeg\H{o} and Weinberger found that Szeg\H{o}'s proof of
(\ref{II-3}) for simply connected domains in $\mathbb{R}^2$ can be
improved to get the estimate
 \begin{eqnarray}  \label{II-5}
 \frac{1}{\mu_{1}(\Omega)}+\frac{1}{\mu_{2}(\Omega)}\geq\frac{2A(\Omega)}{\pi p^{2}_{1,1}}
 \end{eqnarray}
for such domains. Brasco and Pratelli \cite{BP} made a quantitative
improvement of (\ref{II-3}) -- for any bounded domain
$\Omega\subset\mathbb{R}^n$ with smooth boundary, they have proven
 \begin{eqnarray*}
 w_{n}^{2/n}p^{2}_{n/2,1} - \mu_{1}(\Omega)|\Omega|^{2/n}\geq
 c(n)\mathcal{A}(\Omega),
 \end{eqnarray*}
where $c(n)$ is a positive constant depending only on $n$, and
$\mathcal{A}(\Omega)$ is the so-called \emph{Fraenkel asymmetry}
defined by
\begin{eqnarray*}
\mathcal{A}(\Omega):=\inf\left\{\frac{|\Omega\triangle
B|}{|\Omega|}\Bigg{|}B~\mathrm{is~a~ball~in}~\mathbb{R}^{n}~\mathrm{such~that}~|\Omega|=|B|\right\},
\end{eqnarray*}
with $\Omega\triangle B$ the symmetric difference of $\Omega$ and
$B$. An interesting quantitative improvement of (\ref{II-5})
obtained by Nadirashvilli \cite{NN} states that there exists a
constant $C>0$ such that for every smooth simply connected bounded
open set $\Omega\subset\mathbb{R}^2$, it holds
 \begin{eqnarray*}
\frac{1}{|\Omega|}\left(\frac{1}{\mu_{1}(\Omega)}+\frac{1}{\mu_{2}(\Omega)}\right)-\frac{1}{|B|}
\left(\frac{1}{\mu_{1}(B)}+\frac{1}{\mu_{2}(B)}\right)\geq\frac{\mathcal{A}(\Omega)^2}{C},
 \end{eqnarray*}
with $B$ any disk in $\mathbb{R}^2$. Is it possible to improve
(\ref{II-5}) to the higher dimensional case? The answer is
affirmative. In fact, for any bounded domain
$\Omega\subset\mathbb{R}^n$ (with smooth boundary), Ashbaugh and
Benguria \cite{AB} obtained the estimate
\begin{eqnarray} \label{II-6}
\frac{1}{\mu_{1}(\Omega)}+\frac{1}{\mu_{2}(\Omega)}+\cdots+\frac{1}{\mu_{n}(\Omega)}\geq\frac{n}{n+2}\left(\frac{|\Omega|}{w_{n}}\right)^{2/n}.
\end{eqnarray}
Some interesting generalizations to (\ref{II-6}) have been done --
see, e.g., \cite{LMW,CYX}. Based on the estimate (\ref{II-6}),
Ashbaugh and Benguria \cite{AB} proposed an important open problem
as follows:
\begin{itemize}
\item \textbf{Conjecture I}. (\cite{AB}) For any bounded domain
$\Omega$ with smooth boundary in $\mathbb{R}^n$, we have
 \begin{eqnarray*}
\frac{1}{\mu_{1}(\Omega)}+\frac{1}{\mu_{2}(\Omega)}+\cdots+\frac{1}{\mu_{n}(\Omega)}\geq\frac{n}{p^{2}_{n/2,1}}\left(\frac{|\Omega|}{w_n}\right)^{2/n},
 \end{eqnarray*}
 with equality holding if and only if $\Omega$ is a ball in
 $\mathbb{R}^{n}$.
\end{itemize}
\textbf{Conjecture I} is still open until now. Recently,  Xia-Wang
\cite{XW} gave a very important progress to this celebrated
conjecture, and actually they proved that
\begin{eqnarray} \label{II-7}
\frac{1}{\mu_{1}(\Omega)}+\frac{1}{\mu_{2}(\Omega)}+\cdots+\frac{1}{\mu_{n-1}(\Omega)}\geq\frac{n-1}{p^{2}_{n/2,1}}\left(\frac{|\Omega|}{w_n}\right)^{2/n}
 \end{eqnarray}
holds for any bounded domain $\Omega\subset\mathbb{R}^n$ with smooth
boundary, where the equality holds if and only if $\Omega$ is a ball
in $\mathbb{R}^{n}$. Clearly, the isoperimetric inequality
(\ref{II-7}) gives a partial answer to \textbf{Conjecture I} and
also supports its validity. It is not hard to see that our
conclusion in Theorem \ref{theo-1} here covers Xia-Wang's spectral
isoperimetric inequality (\ref{II-7}) as a special case
(corresponding to $\phi=const.$).
 }
\end{remark}

\begin{theorem} \label{theo-2}
Assume that the function $\phi$ satisfies \textbf{Property I}. Let
$\Omega$ be a bounded domain in $\mathbb{H}^n$, and let $B_{R}(o)$
be a geodesic ball of radius $R$ and centered at the
 origin $o$
 of
 $\mathbb{H}^{n}$ such that $|\Omega|_{\phi}=|B_{R}(o)|_{\phi}$.
 Then
 \begin{eqnarray} \label{II-1-1}
\frac{1}{\mu_{1,\phi}(\Omega)}+\frac{1}{\mu_{2,\phi}(\Omega)}+\cdots+\frac{1}{\mu_{n-1,\phi}(\Omega)}\geq\frac{n-1}{\mu_{1,\phi}(B_{R}(o))}.
\end{eqnarray}
The equality case holds if and only if $\Omega$ is isometric to the
geodesic ball $B_{R}(o)$.
\end{theorem}

Similarly, by applying the sequence (\ref{sequence-1}),  from
(\ref{II-1-1}) one has:

\begin{corollary} \label{coro-2}
Under the assumptions of Theorem \ref{theo-2}, we have
\begin{eqnarray} \label{II-2-2}
\mu_{1,\phi}(\Omega)\leq\mu_{1,\phi}(B_{R}(o)),
\end{eqnarray}
with equality holding if and only if $\Omega$ is isometric to
$B_{R}(o)$. That is to say, among all bounded domains in
$\mathbb{H}^n$ having the same weighted volume, the geodesic ball
$B_{R}(o)$ maximizes the first nonzero Neumann eigenvalue of the
Witten-Laplacian, provided the function $\phi$ satisfies
\textbf{Property I}.
\end{corollary}

\begin{remark}
\rm{ Similar to the Euclidean case, the spectral isoperimetric
inequality (\ref{II-2-2}) in Corollary \ref{coro-2} has already been
proven by the authors in \cite{CM} by suitably constructing the
trial function. However, we still wish to list it here to show the
close relation between (\ref{II-1-1}) and (\ref{II-2-2}), and show
of course  the significance of the spectral isoperimetric inequality
(\ref{II-1-1}) as well.  }
\end{remark}

\begin{remark}
\rm{ (1) Similar to (1) of Remark \ref{remark1-4}, except the
situation that $\phi$ is a constant function, one also needs to give
the information of the center for the geodesic ball $B_{R}(o)$
mentioned in Theorem \ref{theo-2} and Corollary \ref{coro-2}. \\
 (2) When investigating spectral isoperimetric inequalities (\ref{II-1-1})-(\ref{II-2-2}), there is no essential difference between $\mathbb{H}^{n}$ and
 a hyperbolic $n$-space with constant curvature not equal to $-1$.\\
(3) Ashbaugh and Benguria \cite{AB} also proposed another important
open problem as follows:
\begin{itemize}
\item \textbf{Conjecture II}. (\cite{AB}) Let
$\mathbb{M}^{n}(\kappa)$ be an $n$-dimensional complete simply
connected Riemannian manifold of constant sectional curvature
$\kappa\in\{1,-1\}$, and $\Omega$ be a bounded domain in
$\mathbb{M}^{n}(\kappa)$ which is contained in a hemisphere in the
case that $\kappa=1$. Let $B_{\Omega}$ be a geodesic ball in
$\mathbb{M}^{n}(\kappa)$ such that $|\Omega|=|B_{\Omega}|$. Then
\begin{eqnarray*}
\frac{1}{\mu_{1}(\Omega)}+\frac{1}{\mu_{2}(\Omega)}+\cdots+\frac{1}{\mu_{n}(\Omega)}\geq\frac{n}{\mu_{1}(B_{\Omega})},
 \end{eqnarray*}
 with equality holding if and only if $\Omega$ is isometric to $B_{\Omega}$.
\end{itemize}
\textbf{Conjecture II} is still open until now. Recently,  Xia-Wang
\cite{XW} also gave a very important progress to this celebrated
conjecture, and actually they proved  that
\begin{eqnarray} \label{II-7-1}
\frac{1}{\mu_{1}(\Omega)}+\frac{1}{\mu_{2}(\Omega)}+\cdots+\frac{1}{\mu_{n-1}(\Omega)}\geq\frac{n-1}{\mu_{1}(B_{\Omega})}
 \end{eqnarray}
holds for any bounded domain $\Omega\subset\mathbb{H}^n$ with smooth
boundary, and for a geodesic ball $B_{\Omega}\subset\mathbb{H}^n$
with $|\Omega|=|B_{\Omega}|$. Moreover, the equality in
(\ref{II-7-1}) holds if and only if $\Omega$ isometric to
$B_{\Omega}$ in $\mathbb{H}^{n}$. Clearly, the isoperimetric
inequality (\ref{II-7-1}) gives a partial answer to
\textbf{Conjecture II} and also supports its validity. It is not
hard to see that our conclusion in Theorem \ref{theo-2} here covers
Xia-Wang's spectral isoperimetric inequality (\ref{II-7-1}) as a
special case (corresponding to $\phi=const.$). }
\end{remark}

\begin{remark}
\rm{ Readers who have interest in this topic might want to know
``\emph{except the Euclidean case and the hyperbolic case, what else
can we expect?}". In fact, after we finished the first version of
this paper at the beginning of 2024, the following related
conclusions have been obtained successfully and quickly. They are:
\begin{itemize}
\item (\cite{DMCW}) For an $n$-dimensional ($n\geq2$) Riemannian manifold $M^{n}$
with non-positive sectional curvature, assume further that the Ricci
curvature of $M^{n}$ satisfies $\mathrm{Ric}(M^{n})\geq(n-1)K$ for
some $K\leq0$. Let $\Omega\subset M^{n}$ be a bounded domain with
smooth boundary, and $\phi$ be a real-valued smooth function defined
on $M^{n}$. Denote by $\Omega^{\ast}\subset\mathbb{R}^{n}$ a
Euclidean $n$-ball having the same volume as $\Omega$, i.e.
$|\Omega^{\ast}|=|\Omega|$, and denote by $\phi^{\ast}$ the
spherically symmetric decreasing rearrangement of $\phi$, which is
defined on $\Omega^{\ast}$. If $\phi^{\ast}$ is convex, then for the
eigenvalue problem (\ref{eigen-1}) one has
\begin{eqnarray} \label{II-1-1x}
\frac{1}{\mu_{1,\phi}(\Omega)}+\frac{1}{\mu_{2,\phi}(\Omega)}+\cdots+\frac{1}{\mu_{n-1,\phi}(\Omega)}\geq\frac{1}{C^{2}}\frac{n-1}{\mu_{1,\phi^{\ast}}(\Omega^{\ast})},
\end{eqnarray}
where
$C=\left(\frac{S_{K}(\mathrm{diam}(\Omega))}{\mathrm{diam}(\Omega)}\right)^{n-1}$,
$\mathrm{diam}(\Omega)$ denotes the diameter of $\Omega$, and the
function $S_{K}(\cdot)$ is defined by
\begin{eqnarray*}
S_{K}(t)=\left\{
\begin{array}{llll}
\frac{\sin\sqrt{K}t}{\sqrt{K}}, &  \quad K>0,\\
 t, & \quad K=0, \\
\frac{\sinh\sqrt{-K}t}{\sqrt{-K}}, & \quad K<0.
\end{array}
\right.
\end{eqnarray*}
Applying the sequence (\ref{sequence-1}), from (\ref{II-1-1x}), it
follows that
\begin{eqnarray} \label{extra-1}
\mu_{1,\phi}(\Omega)\leq\left(\frac{S_{K}(\mathrm{diam}(\Omega))}{\mathrm{diam}(\Omega)}\right)^{2(n-1)}\mu_{1,\phi^{\ast}}(\Omega^{\ast}).
\end{eqnarray}
Especially, if $\phi=const.$ is a constant function, then
$\phi^{\ast}$ would be a constant function also, and
correspondingly, isoperimetric inequalities
(\ref{II-1-1x})-(\ref{extra-1}) directly become
\begin{eqnarray*}
\frac{1}{\mu_{1}(\Omega)}+\frac{1}{\mu_{2}(\Omega)}+\cdots+\frac{1}{\mu_{n-1}(\Omega)}\geq\frac{1}{C^{2}}\frac{n-1}{\mu_{1}(\Omega^{\ast})},\\
\mu_{1}(\Omega)\leq\left(\frac{S_{K}(\mathrm{diam}(\Omega))}{\mathrm{diam}(\Omega)}\right)^{2(n-1)}\mu_{1}(\Omega^{\ast}).~~~~~
\end{eqnarray*}
These two estimates have already been shown in \cite{CH,WK} (as a
special case). BTW, as pointed out in \cite[Remarks 1.7 and
1.8]{DMCW}, one knows: (i) if $K=0$, then the manifold $M^{n}$
considered (in \cite{DMCW}) degenerates into a Ricci-flat manifold
(i.e. $\mathrm{Ric}(M^n)=0$), and consequently,
$S_{K}(\mathrm{diam}(\Omega))=\mathrm{diam}(\Omega)$, $C=1$,
 \begin{eqnarray} \label{extra-2}
\mu_{1,\phi}(\Omega)\leq\mu_{1,\phi^{\ast}}(\Omega^{\ast}),
 \end{eqnarray}
 whose
special case is $\mu_{1}(\Omega)\leq\mu_{1}(\Omega^{\ast})$
(corresponding to $\phi=const.$). Clearly, the estimate
(\ref{extra-2}) can somehow be regarded as a Szeg\"{o}-Weinberger
type spectral isoperimetric inequality under the geometric
constraint $|\Omega^{\ast}|=|\Omega|$ (i.e. the volume is fixed);
(ii) if $M^{n}$ was chosen to be $\mathbb{R}^n$, then
$S_{K}(\mathrm{diam}(\Omega))=\mathrm{diam}(\Omega)$, $C=1$, and the
estimate (\ref{II-1-1x}) becomes
\begin{eqnarray} \label{II-1-2}
\frac{1}{\mu_{1,\phi}(\Omega)}+\frac{1}{\mu_{2,\phi}(\Omega)}+\cdots+\frac{1}{\mu_{n-1,\phi}(\Omega)}\geq\frac{n-1}{\mu_{1,\phi^{\ast}}(\Omega^{\ast})}.
\end{eqnarray}
Specially, if $\phi=const.$ is a constant function, then
(\ref{II-1-2}) degenerates into
\begin{eqnarray*}
\frac{1}{\mu_{1}(\Omega)}+\frac{1}{\mu_{2}(\Omega)}+\cdots+\frac{1}{\mu_{n-1}(\Omega)}\geq\frac{n-1}{\mu_{1}(\Omega^{\ast})}
\end{eqnarray*}
(under the constraint $|\Omega^{\ast}|=|\Omega|$), which is exactly
the isoperimetric inequality (\ref{II-7}). Hence, comparing with
Theorem \ref{theo-1} here, the isoperimetric inequality
(\ref{II-1-1x}), which is walking on another direction and using a
different geometric constraint, can also be seen as a generalization
to Xia-Wang's estimate (\ref{II-7}).

\item (\cite{NMCW}) Denote by $\mathbb{S}^{n}_{+}$ the
$n$-dimensional unit\footnote{~When investigating spectral
isoperimetric inequality (\ref{extra-sph-2}) and related
isoperimetric inequalities on hemispheres, there is no essential
difference between $\mathbb{S}^{n}_{+}$ and
 an $n$-hemisphere with radius not equal to $1$.} hemisphere. Let
$\Omega\subset\mathbb{S}^{n}_{+}$ be bounded domains (on
$\mathbb{S}^{n}_{+}$) with smooth boundary and fixed volume (i.e.
$|\Omega|=const.$), $o$ be the vortex of $\mathbb{S}^{n}_{+}$ (i.e.
in this setting, $\mathbb{S}^{n}_{+}$ was treated as a spherical
cap\footnote{~As mentioned in \cite[Remark 1.6]{NMCW},
$\mathbb{S}^{n}_{+}$ can be seen as a warped product manifold
$\left[0,\frac{\pi}{2}\right]\times_{\sin t}\mathbb{S}^{n-1}$ with
the metric $dt^{2}+\sin^{2}t\cdot g_{\mathbb{S}^{n-1}}$, where
$g_{\mathbb{S}^{n-1}}$ is the round metric on $\mathbb{S}^{n-1}$. In
fact, this kind of warped product manifolds is called spherically
symmetric manifolds. For the notion and some fundamental properties
of warped product manifolds (or more specially, spherically
symmetric manifolds), see e.g. \cite{MDW,PP} for details.}), and
$B_{R}(o)\subset\mathbb{S}^{n}_{+}$ be the geodesic ball with center
$o$ and radius $R$ such that $|B_{R}(o)|=|\Omega|$. If
$|\Omega\setminus B_{R}(o)|+|\Omega|<|\mathbb{S}^{n}_{+}|$, then for
the eigenvalue problem (\ref{eigen-2}) one has
\begin{eqnarray} \label{extra-sph-2}
&&\left(\mu_{1}(B_{R}(o))-\frac{n-1}{\sum_{i=1}^{n-1}\mu_{i}(\Omega)}\right)\int_{B_{R}(o)}f^{2}dv\nonumber\\
&&\qquad \geq(n-1)\left[\int_{B_{R}(o)\setminus
B_{1}}\frac{f^{2}}{\sin^{2}(t)}dv-\int_{B_{2}\setminus
B_{R}(o)}\frac{f^{2}}{\sin^{2}(t)}dv\right]\nonumber\\
&&\qquad
=(n-1)\left[\int_{R_1}^{R}f^{2}(t)\sin^{n-3}(t)dt-f^{2}(R)\int_{R}^{R_2}\sin^{n-3}(t)dt\right],
\end{eqnarray}
with equality holding if and only if $\Omega=B_{R}(o)$, where
$B_{1}$, $B_{2}$ are geodesic balls on $\mathbb{S}^{n}_{+}$
satisfying separately the volume constraints
 \begin{eqnarray*}
|\Omega\cap B_{R}(o)|=|B_1| \quad \mathrm{and}\quad |\Omega\setminus
B_{R}(o)|=|B_{2}\setminus B_{R}(o)|,
\end{eqnarray*}
the function $f(t)$ is defined by
\begin{eqnarray*}
f(t)=\left\{
\begin{array}{lll}
T(t), &  0\leq t<R,\\
T(R), & t\geq R,
\end{array}
\right.
\end{eqnarray*}
with $T(t)$ the solution to the following system
\begin{eqnarray} \label{extra-sph-1}
\left\{
\begin{array}{lll}
T''(t)+\frac{(n-1)\cos t}{\sin t}T'(t)+\left(\mu_{1}(B_{R}(o))-\frac{n-1}{\sin^{2}t}\right)T(t)=0, \\
\\
  T(0)=0,~~T'(R)=0,~~T'|_{[0,R)}\neq0.
\end{array}
\right.
\end{eqnarray}
Here $T(0)=0$ is made to assure the smoothness of $T(t)$. It is not
hard to see that the eigenfunction $u(t,\xi)$ of the first nonzero
Neumann eigenvalue $\mu_{1}(B_{R}(o))$ has the form
$u(x,\xi)=T(t)G(\xi)$, $\xi\in\mathbb{S}^{n-1}$, where $G(\xi)$
should be the restriction of homogeneous harmonic polynomials to
$\mathbb{S}^{n-1}$. That is to say, the solution $T(t)$ to the
system (\ref{extra-sph-1}) is the radial part of the eigenfunction
$u(t,\xi)$. As mentioned in \cite[Remark 1.7]{NMCW}, since
$f(t)/\sin t$ is monotone non-increasing, one can get the fact that
RHS of (\ref{extra-sph-2}) is nonnegative, which directly implies
 \begin{eqnarray*}
\frac{1}{\mu_{1}(\Omega)}+\frac{1}{\mu_{2}(\Omega)}+\cdots+\frac{1}{\mu_{n-1}(\Omega)}\geq\frac{n-1}{\mu_{1}(B_{R}(o))}.
 \end{eqnarray*}
This is exactly the estimate in \cite[Theorem 1.1]{BBC}. Our main
approach in \cite{NMCW} for deriving (\ref{extra-sph-2}) is the one
used in Section \ref{S3} here. So, for bounded domains
$\Omega\subset\mathbb{S}^{n}_{+}$ satisfying $|\Omega\setminus
B_{R}(o)|+|\Omega|<|\mathbb{S}^{n}_{+}|$, our estimate
(\ref{extra-sph-2}) is sharper than the one in \cite[Theorem
1.1]{BBC}.
\end{itemize}
We have attempted to improve the estimates in \cite{BBC, NMCW} to
the case of Witten-Laplacian, but so far no ideal results have been
obtained yet. More precisely, if one primarily uses the approach in
\cite{BBC} or \cite{NMCW}, some accurate estimates involving
eigenfunctions of the first nonzero Neumann eigenvalue cannot work
any more except the situation that the weighted function $\phi$ is a
constant function.

In sum, comparing with our main results in \cite{DMCW, NMCW}, one
can see that spectral isoperimetric inequalities of the
Witten-Laplacian stated in Theorems \ref{theo-1} and \ref{theo-2}
for bounded domains in $\mathbb{R}^n$ or $\mathbb{H}^n$ are much
cleaner but much less restrictive. Hence, we wish only to deal with
the Euclidean case and the hyperbolic case in this paper. }
\end{remark}

Based on the deriving process of our main conclusions in Theorems
\ref{theo-1} and \ref{theo-2}, we would like to propose the
following two open problems, which we think it should be suitable to
call them \emph{\textbf{the Ashbaugh-Benguria type conjecture}}.

\textbf{Question A}. Consider the eigenvalue problem (\ref{eigen-1})
with choosing $M^n$ to be $M^{n}=\mathbb{R}^n$, and assume that the
function $\phi$ satisfies \textbf{Property I}. Let $\Omega$ be a
bounded domain with smooth boundary in
 $\mathbb{R}^n$, and let $B_{R}(o)$ be a ball of radius $R$ and centered at the
 origin $o$
 of
 $\mathbb{R}^{n}$ such that $|\Omega|_{\phi}=|B_{R}(o)|_{\phi}$.
 Then
  \begin{eqnarray*}
\frac{1}{\mu_{1,\phi}(\Omega)}+\frac{1}{\mu_{2,\phi}(\Omega)}+\cdots+\frac{1}{\mu_{n-1,\phi}(\Omega)}+\frac{1}{\mu_{n,\phi}(\Omega)}\geq\frac{n}{\mu_{1,\phi}(B_{R}(o))}.
\end{eqnarray*}
The equality case holds if and only if $\Omega$ is the ball
$B_{R}(o)$.

\textbf{Question B}. Consider the eigenvalue problem (\ref{eigen-1})
with choosing $M^n$ to be $M^{n}=\mathbb{H}^n$, and assume that the
function $\phi$ satisfies \textbf{Property I}. Let $\Omega$ be a
bounded domain with smooth boundary in
 $\mathbb{H}^n$, and let $B_{R}(o)$ be a geodesic ball of radius $R$ and centered at the
 origin $o$
 of
 $\mathbb{H}^{n}$ such that $|\Omega|_{\phi}=|B_{R}(o)|_{\phi}$.
 Then
  \begin{eqnarray*}
\frac{1}{\mu_{1,\phi}(\Omega)}+\frac{1}{\mu_{2,\phi}(\Omega)}+\cdots+\frac{1}{\mu_{n-1,\phi}(\Omega)}+\frac{1}{\mu_{n,\phi}(\Omega)}\geq\frac{n}{\mu_{1,\phi}(B_{R}(o))}.
\end{eqnarray*}
The equality case holds if and only if $\Omega$ is isometric to
$B_{R}(o)$.

\begin{remark}
\rm{Obviously,  Theorems \ref{theo-1} and \ref{theo-2} give a
partial answer to the Ashbaugh-Benguria type conjecture and also
support its validity.
 }
\end{remark}

\begin{remark}
\rm{
 After we posted this paper on arXiv with
 the ID arXiv:2403.08070v2 in March 2024, several colleagues
 communicated with us and thought that maybe the assumption on the
 weighted function $\phi$ is restrictive to get the main conclusions
 given in this paper. So, here we prefer to give again a clear explanation to
 show that the assumption for $\phi$ is necessary and reasonable. In
 fact, if readers have enough patience, they would also find the answer
 in the \textbf{Introduction} part of our another related recent work
 \cite{CM}. Moreover, they could also get the answer through
 carefully checking the details of this paper.

One might have an illusion that since the eigenvalue problem
(\ref{eigen-1}) can be obtained from the classical free membrane
problem (\ref{eigen-2}) by replacing the Laplacian $\Delta$ by the
Witten-Laplacian $\Delta_\phi$, it might be easy and direct to
derive conclusions for eigenvalues of the problem (\ref{eigen-1})
directly from the Neumann eigenvalue problem (\ref{eigen-2}).
Readers would see that this thought is a little bit \textbf{naive}
through only a simple example. In fact, for the free membrane
problem (\ref{eigen-2}), if $\Omega$ is a unit disk in the plane
$\mathbb{R}^2$, it is well-known that the Neumann eigenvalues are
 \begin{eqnarray*}
 (j'_{0,k})^{2}, ~~k\geq 1; \qquad (j'_{n,k})^{2}, ~~n,k\geq 1~~\mathrm{(double~
 eigenvalues)}
 \end{eqnarray*}
whose corresponding eigenfunctions in polar coordinates $(r,\theta)$
are given by
\begin{eqnarray*}
J_{0}(j'_{0,k}r), ~~k\geq 1;  \qquad J_{n}(j'_{n,k}r)\cos n\theta,~
J_{n}(j'_{n,k}r)\sin n\theta, ~~n,k\geq 1,
\end{eqnarray*}
 where $j'_{n,k}$ denotes the $k$-th zero of $J'_{n}$ (i.e., the derivative of the Bessel function
 $J_{n}$).
However, for the eigenvalue problem (\ref{eigen-1}) and under the
assumption that $\Omega$ is a unit disk in $\mathbb{R}^2$, (except
the trivial case $\phi=const.$) it is impossible to compute
\emph{explicitly} all the eigenvalues of the Witten-Laplacian even
if $\phi$ is a radial function.

In the proof of Theorems \ref{theo-1} and \ref{theo-2}, one might
see that the information about the first non-zero Neumann eigenvalue
and its eigenfunctions of the Laplacian on model domains (i.e.,
geodesic balls in space forms) are important for the construction of
desired trial functions. That is to say, in order to get Theorems
\ref{theo-1} and \ref{theo-2} we cannot avoid Lemma \ref{lemma2-1}
in Section \ref{S2}, where one would see that the assumption on the
radial property, the monotonicity and the convexity for the weighted
function $\phi$ is necessary.

 Generally, there do exist some close relationships between geometric
 isoperimetric inequalities and spectral isoperimetric inequalities
 -- readers can immediately get this impression through the
 classical Faber-Krahn inequality. In fact, the Faber-Krahn
 inequality tells us that (see, e.g., \cite[Chapter IV]{IC}):
\begin{itemize}
\item We are given a complete $n$-dimensional ($n\geq2$) Riemannian
manifold $M^{n}$ and for a fixed $\kappa\in\mathbb{R}$, the
complete, simply connected, $n$-dimensional space form
$\mathbb{M}^{n}(\kappa)$ of constant sectional curvature $\kappa$.
To each open set $\Omega$, consisting of a finite disjoint union of
regular domains in $M^{n}$, associate the geodesic disk
$\mathcal{D}$ in $\mathbb{M}^{n}(\kappa)$ satisfying
 \begin{eqnarray}  \label{ADD-1}
 |\Omega|=|\mathcal{D}|.
 \end{eqnarray}
If $\kappa>0$ then only consider those $\Omega$ for which
$|\Omega|<|\mathbb{M}^{n}(\kappa)|$. If, for all such $\Omega$ in
$M^{n}$, equality (\ref{ADD-1}) implies the isoperimetric inequality
 \begin{eqnarray}  \label{ADD-2}
|\partial\Omega|\geq|\partial\mathcal{D}|,
 \end{eqnarray}
 with equality in (\ref{ADD-2}) if and only if $\Omega$ is isometric
 to $\mathcal{D}$, then we also have, for every normal domain
 $\Omega$ in $M^{n}$, that equality (\ref{ADD-1}) implies the
 inequality
  \begin{eqnarray}  \label{ADD-3}
 \lambda_{1}(\Omega)\geq\lambda_{1}(\mathcal{D}),
  \end{eqnarray}
 with equality in (\ref{ADD-3}) if and only if $\Omega$ is isometric
 to $\mathcal{D}$, where $\lambda_{1}(\cdot)$ stands for the first
 Dirichlet eigenvalue of the Laplacian on the prescribed bounded domain.
\end{itemize}
It is well-known that in the Euclidean $n$-space, under the volume
constraint (\ref{ADD-1}) one has the geometric isoperimetric
inequality (\ref{ADD-2}) immediately. Then by the above Faber-Krahn
inequality, one knows:
\begin{itemize}
\item \emph{Among all bounded domains in $\mathbb{R}^n$ having the
same volume, Euclidean balls minimize the first Dirichlet eigenvalue
of the Laplacian.}
\end{itemize}
From the Faber-Krahn inequality, one knows that under the volume
constraint (\ref{ADD-1}), the geometric isoperimetric inequality
(\ref{ADD-2}) makes an important role in the derivation process.
What about the Witten-Laplacian case? Does some weighted geometric
isoperimetric inequality play an important role also? The answer is
affirmative. Now, we prefer to use some content in the
\textbf{Introduction} part of our another work
 \cite{CM} to show again that the assumption on the weighted
 function $\phi$ is necessary and reasonable.  Given a positive function $h$
in $\mathbb{R}^n$, $n\geq2$, one can define the weighted perimeter
and weighted volume of a set $A\subset\mathbb{R}^n$ of locally
finite perimeter as
 \begin{eqnarray*}
 \mathrm{Per}(A)=\int_{\partial A}hd\mathcal{H}^{n-1}, \qquad
 \mathrm{Vol}(A)=\int_{A}hd\mathcal{H}^{n},
 \end{eqnarray*}
where following the usage of notations in \cite{GRC},
$\mathcal{H}^m$ indicates the $m$-dimensional Hausdorff measure, and
$\partial A$ denotes the essential boundary of $A$. Such positive
function $h$ is called a \emph{density} on $\mathbb{R}^n$. If one
fixes a positive weighted volume $m>0$, does there exist a set
$A\subset\mathbb{R}^n$ such that $\mathrm{Vol}(A)=m$ and
 \begin{eqnarray*}
 \mathrm{Per}(A)=\inf\limits_{Q\subset\mathbb{R}^{n},\mathrm{Vol}(Q)=m}\mathrm{Per}(Q)?
 \end{eqnarray*}
Rosales, Ca\~{n}ete, Bayle and Morgan considered this problem and
gave a partial answer that in $\mathbb{R}^n$ with the density
$e^{c|x|^2}$, $c>0$, round balls about the origin uniquely minimize
perimeter for given volume (see \cite[Theorem 5.2]{RCBM}). Moreover,
they showed that for any radial, smooth density $h=e^{f(|x|)}$,
balls around the origin are stable\footnote{~Here ``\emph{stable}"
means that $\mathrm{Per}''(0)\geq0$ under smooth, volume-conserving
variations.} if and only $f$ is convex (\cite[Theorem 3.10]{RCBM}).
This fact motivates the following conjecture (3.12 in their
article), first stated by Kenneth Brakke:
\begin{itemize}
\item (Log-Convex Density Conjecture) \emph{In $\mathbb{R}^n$ with a smooth, radial, log-convex\footnote{~Clearly, for a density $h$ here, the log-convex assumption means $(\log h)''
\geq0$.} density, balls around the origin provide isoperimetric
regions of any given volume.}
\end{itemize}
Chambers \cite[Theorem 1.1]{GRC} recently obtained a recent
breakthrough and gave an answer to the above conjecture as follows:
 \begin{itemize}
 \item  \emph{Given a density $h(x)=e^{f(|x|)}$ on $\mathbb{R}^n$ with $f$ smooth, convex and even, balls around the origin are isoperimetric regions with
respect to weighted perimeter and volume.}
 \end{itemize}
Very recently, using Chambers' above result we have successfully
obtained the Faber-Krahn type inequality for the Witten-Laplcian in
the Euclidean space -- see \cite[Theorem 1.1]{CM} for details. In
our setting of this paper, the density is of the form $e^{-\phi}$,
and then the assumption for the weighted function in Log-Convex
Density Conjecture implies that $\phi$ is radial and
$\left(\log(e^{-\phi})\right)''=-\phi''\geq0$. Based on the
connection of Dirichlet eigenvalues and Neumann eigenvalues, it
should be feasible and reasonable to get the potential spectral
isoperimetric inequalities for the Neumann eigenvalues of the
Witten-Laplacian if the weighted function $\phi$ is radial and
convex. Luckily, this expect has been done here and this is actually
the meaning and the value of our paper. }
\end{remark}

This paper is organized as follows. By suitably constructing trial
functions, we successfully give a proof to Theorems \ref{theo-1} and
\ref{theo-2} in Section \ref{S2}. BTW, since originally the proof of
Theorem \ref{theo-1} is highly similar to that of Theorem
\ref{theo-2}, this leads to the situation that we prefer to unify
those two proofs into a single one, which finally appears as its
present version shown in Section \ref{S2}. A refined result of
Theorem \ref{theo-1} would be given in Section \ref{S3} -- see
Theorem \ref{theo-3} for details.

\section{A proof of Theorems \ref{theo-1} and \ref{theo-2}}
\renewcommand{\thesection}{\arabic{section}}
\renewcommand{\theequation}{\thesection.\arabic{equation}}
\setcounter{equation}{0}  \label{S2}

First, we would like to recall a property of the eigenfunction
corresponding to the first nonzero Neumann eigenvalue of the
Witten-Laplacian on geodesic balls (in space forms) if the function
$\phi$ is radial w.r.t. some chosen point. This property has been
carefully proven in \cite[Appendix]{CM}, and readers can check all
the details therein.

\begin{lemma} \label{lemma2-1}(\cite[Theorem 4.1]{CM})
Assume that $B_{R}(o)$ is a geodesic ball of radius $R$ and centered
at some point $o$ in the $n$-dimensional complete simply connected
Riemannian manifold $\mathbb{M}^{n}({\kappa})$ with constant
sectional curvature $\kappa\in\{-1,0,1\}$, and that  $\phi$ is a
radial function w.r.t. the distance parameter $t:=d(o,\cdot)$, which
is also a non-increasing convex function. Then the eigenfunctions of
the first nonzero Neumann eigenvalue $\mu_{1,\phi}(B_{R}(o))$ of the
Witten-Laplacian
 on $B_{R}(o)$ have the form $T(t)\frac{x_i}{t}$,
 $i=1,2,\cdots,n$, where $T(t)$
satisfies
\begin{eqnarray}  \label{2-1}
\left\{
\begin{array}{ll}
T''+\left(\frac{(n-1)C_{\kappa}}{S_{\kappa}}-\phi'\right)T'+\left(\mu_{1,\phi}(B_{R}(o))-(n-1)S_{\kappa}^{-2}\right)T=0,\\[1mm]
T(0)=0,~T'(R)=0,~T'|_{[0,R)}\neq0.
\end{array}
\right.
\end{eqnarray}
 Here $C_{\kappa}(t)=\left(S_{\kappa}(t)\right)'$ and
 \begin{eqnarray*}
S_{\kappa}(t)=\left\{
\begin{array}{lll}
\sin t,~~&\mathrm{if}~\mathbb{M}^{n}(\kappa)=\mathbb{S}^{n}_{+}, \\
t,~~&\mathrm{if}~\mathbb{M}^{n}(\kappa)=\mathbb{R}^{n},\\
\sinh t,~~&\mathrm{if}~\mathbb{M}^{n}(\kappa)=\mathbb{H}^{n},
\end{array}
\right.
 \end{eqnarray*}
  with $\mathbb{S}^{n}_{+}$ the $n$-dimensional hemisphere of radius $1$.
\end{lemma}

\begin{remark}
\rm{ From \cite[Appendix]{CM}, it is not hard to know that $x_{i}$,
$i=1,2,\cdots,n$, are coordinate functions of the globally defined
orthonormal coordinate system set up in $\mathbb{M}^{n}(\kappa)$. }
\end{remark}

 \textbf{\emph{A proof of Theorems \ref{theo-1} and
\ref{theo-2}}}. Due to the fact that $\phi$ is radial w.r.t. $o$,
one can define a radial function $f$ as follows
 \begin{eqnarray}  \label{f-add}
f(t)=\left\{
\begin{array}{ll}
T(t),~~&\mathrm{if} ~0\leq t \leq R,\\
T(R),~~&\mathrm{if} ~t> R,
\end{array}
\right.
 \end{eqnarray}
where $R$ is the radius of the (geodesic) ball $B_{R}(o)$ satisfying
the volume constraint $|\Omega|_{\phi}=|B_{R}(o)|_{\phi}$. The
origin $o$ would be chosen as follows: in fact, by the Brouwer fixed
point theorem and using a similar argument to that of Weinberger
given in \cite{HFW}, one can always choose a suitable origin
$o\in\mathrm{hull}(\Omega)$ such that
\begin{eqnarray} \label{2-2}
\int_\Omega f(t) \frac{x_i}{t} d\eta=0, \qquad i=1,2,\cdots,n.
\end{eqnarray}
Denote by $\{e_{1},e_{2},\cdots,e_{n}\}$ the orthonormal basis (of
$\mathbb{R}^n$ or $\mathbb{H}^n$) corresponding to the coordinates
$x_{1},x_{2},\cdots,x_{n}$. Then (\ref{2-2}) can be rewritten as
\begin{eqnarray} \label{2-3}
\int_\Omega \langle x,e_{i}\rangle \frac{f(t)}{t} d\eta=0, \qquad
i=1,2,\cdots,n,
\end{eqnarray}
 with $\langle\cdot,\cdot\rangle$ denoting the inner product. Denote
 by $u_{i}$ the eigenfunction corresponding to the $i$-th Neumann
 eigenvalue $\mu_{i,\phi}$ of the eigenvalue problem
 (\ref{eigen-1}). Then (\ref{2-3})
 implies that
 $$\langle x,e_{i}\rangle \frac{f(t)}{t}\perp u_{0}$$
 in the sense of $L^{2}$-norm w.r.t. the weighted density $d\eta$.
 Our purpose now is to construct suitable trial function $\psi_{i}$
 for the eigenvalue $\mu_{i,\phi}$ such that $\psi_{i}$ is orthogonal to the preceding
 eigenfunctions $u_{0},u_{1},\cdots,u_{i-1}$. That is to say,
 $\psi_{i}\perp\mathrm{span}\{u_{0},u_{1},\cdots,u_{i-1}\}$ in the sense of $L^{2}$-norm w.r.t. the weighted density
 $d\eta$. Define an $n\times n$ matrix $Q=(q_{ij})_{n\times n}$ with
 $q_{ij}$ given by
  \begin{eqnarray*}
q_{ij}:=\int_\Omega \langle x,e_{i}\rangle \frac{f(t)}{t}u_{j}
d\eta, \qquad i,j=1,2,\cdots,n.
  \end{eqnarray*}
Using the orthogonalization of Gram and Schmidt (QR-factorization
theorem), one knows that there exist an upper triangle matrix
$\mathcal{M}=(\mathcal{M}_{ij})_{n\times n}$ and an orthogonal
matrix $U=(a_{ij})_{n\times n}$ such that $\mathcal{M}=UQ$, which
implies
\begin{eqnarray*}
\mathcal{M}_{ij}=\sum\limits_{k=1}^{n}a_{ik}q_{kj}=\int_{\Omega}a_{ik}\langle
x,e_{k}\rangle \frac{f(t)}{t}u_{j}d\eta, \qquad 1\leq j<i\leq n.
\end{eqnarray*}
 Set $e'_{i}=\sum_{k=1}^{n}a_{ik}e_{k}$, $i=1,2,\cdots,n$, and then
\begin{eqnarray} \label{2-5}
\int_\Omega \langle x,e'_{i}\rangle \frac{f(t)}{t}u_{j} d\eta=0
\end{eqnarray}
holds for $j=1,2,\cdots,i-1$ and $i=2,3,\cdots,n$. BTW, it is easy
to see that $\{e'_{1},e'_{2},\cdots,e'_{n}\}$ is also an orthonormal
basis (of $\mathbb{R}^n$ or $\mathbb{H}^n$), which is actually
formed by making an orthogonal transformation to the orthonormal
basis $\{e_{1},e_{2},\cdots,e_{n}\}$. Denote by
$y_{1},y_{2},\cdots,y_{n}$ the coordinate functions corresponding to
the basis $\{e'_{1},e'_{2},\cdots,e'_{n}\}$, that is, $y_{i}=\langle
x,e'_{i}\rangle$. Then from (\ref{2-5}) one has
  \begin{eqnarray}  \label{2-6}
\int_\Omega y_{i}\frac{f(t)}{t}u_{j} d\eta=0, \qquad
j=1,2,\cdots,i-1~ \mathrm{and} ~i=2,3,\cdots,n.
  \end{eqnarray}
 For convention and by the abuse of notations, we prefer to use
 $x_{i}$ as coordinate functions -- based on this, we still write
 $y_{i}$ as $x_{i}$, $i=2,3,\cdots,n$. Then in this setting,
 (\ref{2-6}) can be rewritten as
\begin{eqnarray}  \label{2-6-1}
\int_\Omega x_{i}\frac{f(t)}{t}u_{j} d\eta=0, \qquad
j=1,2,\cdots,i-1~ \mathrm{and} ~i=2,3,\cdots,n.
  \end{eqnarray}
Together with (\ref{2-2}) and (\ref{2-6-1}), one has that there
exists an orthonormal basis $\{e_{1},e'_{2},e'_{3},\cdots,e'_{n}\}$
such that the coordinate functions $x_{1},x_{2},\cdots,x_{n}$
corresponding to this basis satisfy
 \begin{eqnarray} \label{2-7}
\int_\Omega x_{i}\frac{f(t)}{t}u_{j} d\eta=0, \qquad
j=0,1,2,\cdots,i-1~ \mathrm{and} ~i=1,2,3,\cdots,n.
 \end{eqnarray}
Here the eigenfunction $u_{0}$ of the eigenvalue $\mu_{0,\phi}$ can
be chosen as $u_{0}=1/\sqrt{|\Omega|_{\phi}}$. Set in (\ref{2-7})
that
 \begin{eqnarray*}
 \psi_{i}:=x_{i}\frac{f(t)}{t}, \qquad i=1,2,3,\cdots,n,
 \end{eqnarray*}
and then one has
 \begin{eqnarray}  \label{Tadd}
\int_\Omega \psi_{i}u_{j} d\eta=0, \qquad j=0,1,2,\cdots,i-1~
\mathrm{and} ~i=1,2,3,\cdots,n.
 \end{eqnarray}
Hence, our purpose of constructing trial functions $\psi_{i}$,
$i=1,2,3,\cdots,n$, has been achieved. To prove our main
conclusions, we also need the following fact.

\begin{lemma} \label{lemma2-3}
The function $\frac{f(t)}{S_{\kappa}(t)}$ is monotone decreasing in
the bounded domain $\Omega$ with smooth boundary in $\mathbb{R}^n$
(or $\mathbb{H}^n$).
\end{lemma}

\begin{proof}
By (\ref{2-1}) and the definition of the function $f$, we observe
first that
 \begin{eqnarray*}
\lim_{t\rightarrow0}\frac{f(t)}{S_{\kappa}(t)}=f'(0).
 \end{eqnarray*}
Without loss of generality, we may assume $f>0$. Since
 \begin{eqnarray*}
 \frac{d}{dt}\left(\frac{f(t)}{S_{\kappa}(t)}\right)=\frac{f'(t)-\frac{C_{\kappa}(t)}{S_{\kappa}(t)}f(t)}{S_{\kappa}(t)},
\end{eqnarray*}
 similarly, one has
 \begin{eqnarray*}
 \lim_{t\rightarrow0}\left(f'(t)-\frac{C_{\kappa}(t)}{S_{\kappa}(t)}f(t)\right)=0, \qquad~f'(R)-\frac{C_{\kappa}(R)}{S_{\kappa}(R)}f(R)<0.
\end{eqnarray*}
If there exists a point $t_0$ such that
$f'(t_0)-\frac{C_{\kappa}(t_0)}{S_{\kappa}(t_0)}f(t_0) > 0$, then
there exists $t_1$ such that
\begin{eqnarray*}
f'(t_1)-\frac{C_{\kappa}(t_1)}{S_{\kappa}(t_1)}f(t_1)>0,
\end{eqnarray*}
\begin{eqnarray} \label{2-9}
\frac{d}{dt}\left(f'(t)-\frac{C_{\kappa}(t)}{S_{\kappa}(t)}f(t)\right)(t_1)=0.
\end{eqnarray}
 Combining the first equation in (\ref{2-1}) and (\ref{2-9}) yields at point $t_1$ that
  \begin{eqnarray} \label{2-10}
-\frac{nC_{\kappa}}{S_{\kappa}}f'-\mu_{1,\phi}f+\phi'f'+\frac{nf}{S_{\kappa}^2}=0.
  \end{eqnarray}
 Therefore, due to $\phi'\leq0$, it follows from (\ref{2-10}) that
\begin{eqnarray*}
\left(f'-\frac{f}{C_{\kappa} S_{\kappa}}\right)(t_1)\leq 0.
\end{eqnarray*}
So, we have
\begin{eqnarray*}
\left(f'-\frac{ C_{\kappa}}{ S_{\kappa}}f\right)(t_1)&\leq& \left(\frac{f}{S_{\kappa} C_{\kappa}}-\frac{fC_{\kappa}}{S_{\kappa}}\right)(t_1)\\
&=&\left(\frac{f(1-C_{\kappa}^2)}{S_{\kappa} C_{\kappa}}\right)(t_1) \\
&\leq& 0.
\end{eqnarray*}
This is contradict with $\left(f'-\frac{ C_{\kappa}}{
S_{\kappa}}f\right)(t_1)> 0$. Hence, we have
$\frac{d}{dt}(\frac{f(t)}{S_{\kappa}(t)}) < 0$, and then
$\frac{f(t)}{S_{\kappa}(t)}$ is monotone decreasing.
\end{proof}

By the characterization (\ref{chr-1}) and (\ref{Tadd}), one can
obtain
 \begin{eqnarray} \label{2-11}
\mu_{i,\phi}(\Omega)\int_{\Omega} f^2\frac{x_i^2}{t^2} d\eta \leq
\int_{\Omega}\left(f'^2\frac{x_i^2}{t^2} + f^2
\left|\overline{\nabla} \left(\frac{x_i}{t}\right)\right|^2
S_{\kappa}^{-2}(t)\right)d\eta,
 \end{eqnarray}
where $\overline{\nabla}$ is the gradient operator defined on the
unit $(n-1)$-sphere $\mathbb{S}^{n-1}$. By a direct calculation to
(\ref{2-11}), one has
\begin{eqnarray} \label{2-12}
 \int_{\Omega} f^2\frac{x_i^2}{t^2} d\eta
&\leq &\frac{1}{\mu_{i,\phi}(\Omega)} \int_{\Omega}(f')^2\frac{x_i^2}{t^2}d\eta +\frac{1}{\mu_{i,\phi}(\Omega)} \int_{\Omega}f^2 \left|\overline{\nabla}\left(\frac{x_i}{t}\right)\right|^2 S_{\kappa}^{-2}(t) d\eta\nonumber\\
&=&  \frac{1}{\mu_{i,\phi}(\Omega)}  \int_{\Omega \cap B_{R}(o)} (f')^2\frac{x_i^2}{t^2}d\eta +\frac{1}{\mu_{i,\phi}(\Omega)} \int_{\Omega}f^2 \left|\overline{\nabla}\left(\frac{x_i}{t}\right)\right|^2 S_{\kappa}^{-2}(t)d\eta\nonumber\\
&\leq&  \frac{1}{\mu_{i,\phi}(\Omega)}  \int_{B_{R}(o)} (f')^2\frac{x_i^2}{t^2}d\eta +\frac{1}{\mu_{i,\phi}(\Omega)} \int_{\Omega}f^2 \left|\overline{\nabla} \frac{x_i}{t}\right|^2 S_{\kappa}^{-2}(t)d\eta\nonumber\\
&=&  \frac{1}{\mu_{i,\phi}(\Omega)} \frac{1}{n} \int_0^R \int_{\mathbb{S}^{n-1}(1)} (f')^2 S_{\kappa}^{n-1}(t) ~e^{-\phi} dS dt \nonumber\\
 && \qquad \qquad +\frac{1}{\mu_{i,\phi}(\Omega)} \int_{\Omega}f^2 \left|\overline{\nabla}\left(\frac{x_i}{t}\right)\right|^2S_{\kappa}^{-2}(t)d\eta\nonumber\\
&=&  \frac{1}{\mu_{i,\phi}(\Omega)} \frac{1}{n} \int_{B_{R}(o)}
(f')^2 d\eta +\frac{1}{\mu_{i,\phi}(\Omega)} \int_{\Omega}f^2
\left|\overline{\nabla} \left(\frac{x_i}{t}\right)\right|^2
S_{\kappa}^{-2}(t) d\eta,
\end{eqnarray}
where $dS$ stands for the volume element on the $(n-1)$-sphere
$\mathbb{S}^{n-1}(1)$ of radius $1$. By \cite{XW}, one knows
 \begin{eqnarray} \label{2-13}
\sum\limits_{i=1}^{n}\frac{1}{\mu_{i,\phi}(\Omega)}
\left|\overline{\nabla}\left(\frac{x_i}{r}\right)\right|^2 \leq
\sum\limits_{i=1}^{n-1} \frac{1}{\mu_{i,\phi}(\Omega)}.
 \end{eqnarray}
Therefore, combining (\ref{2-12}) with (\ref{2-13}), and then doing
summation over the index $i$ from $1$ to $n$, we can obtain
 \begin{eqnarray}  \label{2-14}
\int_{\Omega} f^2 d\eta\leq\sum_{i=1}^n\frac{1}{n
\mu_{i,\phi}(\Omega)} \int_{B_{R}(o)} (f')^2 d\eta +\sum_{i=1}^{n-1}
\frac{1}{\mu_{i,\phi}(\Omega)} \int_{\Omega}f^2 S_{\kappa}^{-2}(t)
d\eta.
 \end{eqnarray}
On the other hand, still from (\ref{2-13}), one has
 \begin{eqnarray*}
\sum_{i=1}^n \frac{1}{\mu_{n,\phi}(\Omega)} \left|\overline{\nabla}
\left(\frac{x_i}{r}\right)\right|^2 \leq \sum_{i=1}^n
\frac{1}{\mu_{i,\phi}(\Omega)} \left|\overline{\nabla}
\left(\frac{x_i}{r}\right)\right|^2 \leq \sum_{i=1}^{n-1}
\frac{1}{\mu_{i,\phi}(\Omega)},
 \end{eqnarray*}
 which implies
  \begin{eqnarray*}
\frac{1}{n \mu_{n,\phi}(\Omega)} \leq
\frac{1}{n(n-1)}\sum_{i=1}^{n-1} \frac{1}{\mu_{i,\phi}(\Omega)}.
\end{eqnarray*}
Substituting the above inequality into (\ref{2-14}) results in
 \begin{eqnarray}   \label{2-15}
\int_\Omega f^2(t) d\eta \leq \sum_{i=1}^{n-1} \frac{1}{(n-1)
\mu_{i,\phi}(\Omega)} \left[\int_{B_{R}(o)} (f')^2(t) d\eta +
\int_\Omega (n-1)f^2(t)S_{\kappa}^2(t) d\eta\right].
 \end{eqnarray}
Applying Lemma \ref{lemma2-3} and \cite[Lemma 4.4 and Appendix]{CM},
one has
\begin{eqnarray*}
\int_\Omega f^2(t) d\eta \geq \int_{B_{R}(o)} f^2(t) d\eta, \qquad
\int_\Omega \frac{f^2(t)}{S_{\kappa}^2(t)} d\eta \leq
\int_{B_{R}(o)} \frac{f^2(t)}{S_{\kappa}^2(t)} d\eta.
\end{eqnarray*}
Putting the above fact into (\ref{2-15}), we have
\begin{eqnarray*}
\frac{1}{n-1}\sum_{i=1}^{n-1}\frac{1}{\mu_{i,\phi}(\Omega)} \geq
\frac{ \int_{B_{R}(o)} f^2(t) d\eta}{\int_{B_{R}(o)}
\left[(f')^2+(n-1)\frac{f^2(t)}{S_{\kappa}^2(t)}\right] d\eta}
=\frac{1}{\mu_{1,\phi}(B_{R}(o))},
\end{eqnarray*}
which implies (\ref{II-1}) or (\ref{II-1-1}) directly. This
completes the proof of Theorems \ref{theo-1} and \ref{theo-2}.

\section{A sharper estimate}
\renewcommand{\thesection}{\arabic{section}}
\renewcommand{\theequation}{\thesection.\arabic{equation}}
\setcounter{equation}{0}  \label{S3}

In the last section, we would like to give a shaper estimate (for
the sum of the reciprocals of the first $(n-1)$ nonzero Neumann
eigenvalues of the Witten-Laplacian on bounded domains in
$\mathbb{R}^n$) than (\ref{II-1}) shown in Theorem \ref{theo-1}. In
fact, we can prove:

\begin{theorem} \label{theo-3}
Assume that $\Omega$ is a bounded domain in $\mathbb{R}^n$ with
smooth boundary, and that the function $\phi$ satisfies
\textbf{Property I}. Then
 \begin{eqnarray}  \label{3-1}
&&\mu_{1,\phi}(B_{R}(o))-\frac{n-1}{\frac{1}{\mu_{1,\phi}(\Omega)}+\frac{1}{\mu_{2,\phi}(\Omega)}+\cdots+\frac{1}{\mu_{n-1,\phi}(\Omega)}}\nonumber\\
 &&\qquad \qquad \geq \frac{\int_{B_{R}(o)\setminus B_1}\left[(f')^2+(n-1)\frac{f^2}{t^2}\right]d\eta - f^2(R)\int_{B_2\setminus
 B_{R}(o)}\left[(n-1)\frac{1}{t^2}\right]d\eta}{\int_{B_{R}(o)}f^2d\eta},
 \qquad
 \end{eqnarray}
  where (as in Theorem \ref{theo-1}) $B_{R}(o)$ is a ball of radius $R$ and centered at the
 origin $o$
 of
 $\mathbb{R}^{n}$ such that $|\Omega|_{\phi}=|B_{R}(o)|_{\phi}$, $f$ is the function defined by (\ref{f-add}), and $B_1$, $B_2$ are two balls
 centered
  at the origin $o$ and satisfying $|B_1|_\phi=|\Omega\cap B_{R}(o)|_\phi$, $|B_2\setminus B_{R}(o)|_\phi=|\Omega\setminus
  B_{R}(o)|_\phi$, respectively. The equality in (\ref{3-1}) holds
  if and only if $\Omega$ is the ball $B_{R}(o)$.
\end{theorem}

\begin{proof}
 By Lemma
\ref{lemma2-3}, we have $f'-\frac{1}{t}f\leq 0$ and $f'\geq 0$ in
$[0,R]$, which implies
\begin{eqnarray*}
(f')^2 - \frac{f^2}{t^2} \leq 0.
\end{eqnarray*}
Since $|\overline{\nabla} \frac{x_i}{t}|^2=1-\frac{x_i^2}{t^2}$ and
$(f')^2 - \frac{f^2}{t^2} \leq 0$, with the help of trial functions
$\psi_{i}$, $i=1,2,\cdots,n$, constructed in Section \ref{S2} and by
using a similar argument of deriving the inequality (2.31) in
\cite{XW}, it is not hard to get
\begin{eqnarray} \label{3-2}
\frac{n-1}{\sum\limits_{i=1}^{n-1}\frac{1}{\mu_{i,\phi}}}
\int_\Omega f^2 d\eta \leq  \int_\Omega
\left[(f')^2+(n-1)\frac{f^2}{t^2}\right] d\eta.
\end{eqnarray}
Since $f$ is increasing, we can deduce  from \cite[Lemma 4.4 and
Appendix]{CM} by the rearrangement technique the following:
\begin{eqnarray} \label{rigidity}
\int_\Omega f^2 d\eta \geq \int_{B_{R}(o)} f^2 d\eta.
\end{eqnarray}
Putting the above expression into (\ref{3-2}) yields
\begin{eqnarray} \label{3-3}
\frac{n-1}{\sum_{i=1}^{n-1}\frac{1}{\mu_{i,\phi}}} \int_{B_{R}(o)}
f^2 d\eta \leq  \int_\Omega \left[(f')^2+(n-1)\frac{f^2}{t^2}\right]
d\eta.
\end{eqnarray}
Since $f(t)\frac{x_i}{t}$, $i=1,2,\cdots,n$, are the eigenfunctions
corresponding to the eigenvalue $\mu_{1,\phi}(B_{R}(o))$, one can
obtain from the characterization (\ref{chr-2}) that
\begin{eqnarray} \label{3-4}
\mu_{1,\phi}(B_{R}(o)) \int_{B_{R}(o)} f^2 d\eta =  \int_{B_{R}(o)}
\left[(f')^2+(n-1)\frac{f^2}{t^2}\right] d\eta.
\end{eqnarray}
Combining (\ref{3-3}) and  (\ref{3-4}) results in
\begin{eqnarray}  \label{3-5}
&&\left(\mu_{1,\phi}(B_{R}(o)) -
\frac{n-1}{\sum_{i=1}^{n-1}\frac{1}{\mu_{i,\phi}}}\right)\int_{B_{R}(o)}
f^2 d\eta \geq   \nonumber\\
&& \qquad\qquad  \int_{B_{R}(o)}
\left[(f')^2+(n-1)\frac{f^2}{t^2}\right]
d\eta-\int_{\Omega}\left[(f')^2+(n-1)\frac{f^2}{t^2}\right] d\eta.
\end{eqnarray}
On one hand,
\begin{eqnarray} \label{3-6}
\int_{\Omega} \left[(f')^2+(n-1)\frac{f^2}{t^2}\right] d\eta &=&\int_{(\Omega\setminus B_{R}(o)) \cup (\Omega \cap B_{R}(o))}\left[(f')^2+(n-1)\frac{f^2}{t^2}\right] d\eta \nonumber\\
&=&\int_{\Omega \setminus
B_{R}(o)}\left[(f')^2+(n-1)\frac{f^2}{t^2}\right] d\eta + \nonumber\\
 && \qquad \qquad \int_{\Omega\cap B_{R}(o)}\left[(f')^2+(n-1)\frac{f^2}{t^2}\right]
 d\eta.  \qquad \qquad
\end{eqnarray}
By \cite{CM}, it is not hard to show that
$(f')^2+(n-1)\frac{f^2}{t^2}$ is monotone decreasing along the
radial direction of $(\Omega\cap B_{R}(o))\setminus B_1$, which
implies
 \begin{eqnarray}  \label{add-11}
&& \int_{\Omega \cap B_{R}(o)}
\left[(f')^2+(n-1)\frac{f^2}{t^2}\right]
d\eta\nonumber\\
  &=& \int_{\Omega \cap B_{R}(o) \cap B_1}
\left[(f')^2+(n-1)\frac{f^2}{t^2}\right]d\eta+ \nonumber\\
 && \quad  \int_{(\Omega \cap B_{R}(o)) \setminus B_1} \left[(f')^2+(n-1)\frac{f^2}{t^2}\right]d\eta \nonumber\\
&\leq& \int_{\Omega \cap B_{R}(o) \cap B_1}
\left[(f')^2+(n-1)\frac{f^2}{t^2}\right]d\eta+
\nonumber\\
&& \quad
\left[(f')^2(R_1)+(n-1)\frac{f^2(R_1)}{R_1^2}\right]\int_{(\Omega
\cap B_{R}(o)) \setminus B_1}d\eta.  \qquad
 \end{eqnarray}
Similarly, one can obtain
\begin{eqnarray} \label{add-111}
\int_{B_1} \left[(f')^2+(n-1)\frac{f^2}{t^2}\right] d\eta&=&\int_{B_1 \cap \Omega\cap B_{R}(o)} \left[(f')^2+(n-1)\frac{f^2}{t^2}\right]d\eta+\nonumber\\
 && \qquad  \int_{B_1 \setminus (\Omega\cap B_{R}(o))}\left[(f')^2+(n-1)\frac{f^2}{t^2}\right]d\eta \nonumber\\
&\geq& \int_{ B_1 \cap \Omega\cap B_{R}(o)}
\left[(f')^2+(n-1)\frac{f^2}{t^2}\right]d\eta+\nonumber\\
 && \qquad \left[(f'(R_1))^2+(n-1)\frac{f^2(R_1)}{R_1^2}\right]\int_{B_1
\setminus (\Omega\cap B_{R}(o))} d\eta,  \qquad
\end{eqnarray}
where $R_{1}$ is the radius of the ball $B_1$. One has from the
assumption $|\Omega\cap B_{R}(o)|_\phi=|B_1|_\phi$ that
 \begin{eqnarray}  \label{3-8}
\int_{\Omega \cap B_{R}(o)}\left[(f')^2+(n-1)\frac{f^2}{t^2}\right]
d\eta\leq\int_{B_1}\left[(f')^2+(n-1)\frac{f^2}{t^2}\right]d\eta.
 \end{eqnarray}
Since $f(t)=T(R)$ is constant when $t>R$, by a direct calculation
one has
 \begin{eqnarray} \label{add-l111}
&&\int_{\Omega\setminus
B_{R}(o)}\left[(f')^2+(n-1)\frac{f^2}{t^2}\right]
d\eta\nonumber\\
 &=&
\int_{(\Omega\setminus B_{R}(o)) \cap (B_2\setminus B_{R}(o))} \left[(f')^2+(n-1)\frac{f^2}{t^2}\right]d\eta+\nonumber\\
 &&\qquad \int_{(\Omega \setminus B_{R}(o)) \setminus (B_2\setminus B_{R}(o))} \left[(f')^2+(n-1)\frac{f^2}{t^2}\right]d\eta\nonumber\\
&=&\int_{(\Omega\setminus B_{R}(o)) \cap (B_2\setminus B_{R}(o))}
\left[(f')^2+(n-1)\frac{f^2}{t^2}\right]d\eta+\nonumber\\
 && \quad \left[(f'(R))^2+(n-1)\frac{f^2(R)}{R^2}\right]\int_{(\Omega
\setminus B_{R}(o)) \setminus (B_2\setminus B_{R}(o))} d\eta.
 \end{eqnarray}
Similarly, one can get
 \begin{eqnarray} \label{add-ll111}
&& \int_{B_2\setminus
B_{R}(o)}\left[(f')^2+(n-1)\frac{f^2}{t^2}\right]d\eta
\nonumber\\
&=&\int_{(B_2\setminus B_{R}(o)) \cap (\Omega\cap
B_{R}(o))}
 \left[(f')^2+(n-1)\frac{f^2}{t^2}\right]d\eta+ \nonumber\\
  && \qquad \int_{(B_2 \setminus B_{R}(o) )\setminus(\Omega\setminus B_{R}(o))}\left[(f')^2+(n-1)\frac{f^2}{t^2}\right]d\eta \nonumber\\
&=& \int_{(B_2\setminus B_{R}(o)) \cap (\Omega\cap B_{R}(o))}
\left[(f')^2+(n-1)\frac{f^2}{t^2}\right]d\eta+ \nonumber\\
 &&\qquad \left[(f'(R))^2+(n-1)\frac{f^2(R)}{R^2}\right]\int_{(B_2
\setminus B_{R}(o))\setminus(\Omega\setminus B_{R}(o))}d\eta.
 \end{eqnarray}
One has from the assumption $|\Omega\setminus
B_{R}(o)|_\phi=|B_2\setminus B_{R}(o)|_\phi$ that
 \begin{eqnarray} \label{3-13}
\int_{\Omega \setminus
B_{R}(o)}\left[(f')^2+(n-1)\frac{f^2}{t^2}\right] d\eta =
\int_{B_2\setminus B_{R}(o)}
\left[(f')^2+(n-1)\frac{f^2}{t^2}\right]d\eta.
 \end{eqnarray}
 Putting (\ref{3-6})-(\ref{3-13}) into (\ref{3-5}) yields
\begin{eqnarray*}
&&\left(\mu_{1,\phi}(B_{R}(o)) -
\frac{n-1}{\sum_{i=1}^{n-1}\frac{1}{\mu_{i,\phi}}}
\right)\int_{B_{R}(o)} f^2 d\eta \geq   \int_{B_{R}(o)\setminus B_1}
\left[(f')^2+(n-1)\frac{f^2}{t^2}\right] d\eta- \\
&& \qquad \qquad \qquad\qquad\qquad\qquad\qquad \qquad\qquad\qquad
\int_{B_2\setminus B_{R}(o)}
\left[(f')^2+(n-1)\frac{f^2}{t^2}\right] d\eta,
\end{eqnarray*}
which implies (\ref{3-1}) directly by using (\ref{f-add}) first and
then multiplying both sides of the above inequality by
$\left(\int_{B_{R}(o)} f^2 d\eta\right)^{-1}$. The equality case of
the estimate (\ref{3-1}) would follow by (\ref{rigidity}) where the
equality can be attained if and only if $\Omega$ is the ball
$B_{R}(o)$.
\end{proof}

\begin{remark}
\rm{ The estimate (\ref{3-1}) is sharper than (\ref{II-1}) in
Theorem \ref{theo-1}, since the quantity in the RHS of (\ref{3-1})
is nonnegative. This is because
 \begin{eqnarray*}
&&\int_{B_{R}(o)\setminus B_1} \left[(f')^2+(n-1)\frac{f^2}{t^2}\right] d\eta-f^2(R)\int_{B_2\setminus B_{R}(o)}(n-1)\frac{1}{t^2} d\eta\\
&&\qquad \geq\left[(f'(R))^2+(n-1)\frac{f^2(R)}{R^2}\right]\int_{B_{R}(o)\setminus B_1} d\eta-\frac{(n-1)f^2(R)}{R^2}\int_{B_2\setminus B_{R}(o)} d\eta\\
&&\qquad =(n-1)\frac{f^2(R)}{R^2}\left(\int_{B_{R}(o)\setminus B_1} d\eta-\int_{B_2\setminus B_{R}(o)}d\eta\right)\\
&&\qquad = (n-1)\frac{f^2(R)}{R^2}(|B_{R}(o)\setminus B_1|_\phi - |B_2\setminus B_{R}(o)|_\phi)\\
&&\qquad = (n-1)\frac{f^2(R)}{R^2}(|B_{R}(o)|_\phi - |B_1|_\phi - (|B_2|_\phi - |B_{R}(o)|_\phi))\\
&&\qquad = (n-1)\frac{f^2(R)}{R^2}(2|B_{R}(o)|_\phi - (|B_1|_\phi + |B_2|_\phi))\\
&&\qquad = (n-1)\frac{f^2(R)}{R^2}(2|\Omega|_\phi - (|\Omega \cap B_{R}(o)|_\phi + |\Omega\setminus B_{R}(o)|_\phi + |B_{R}(o)|_\phi))\\
&&\qquad =0.
\end{eqnarray*}
}
\end{remark}

\begin{remark}
\rm{ It is not hard to know that using a similar argument to the one
used in Section \ref{S3}, a shaper estimate (for the sum of the
reciprocals of the first $(n-1)$ nonzero Neumann eigenvalues of the
Witten-Laplacian on bounded domains in $\mathbb{H}^n$) than
(\ref{II-1-1}) can be obtained, and we wish to leave this as an
exercise for readers who have interest in this topic. }
\end{remark}

\section*{Acknowledgments}
\renewcommand{\thesection}{\arabic{section}}
\renewcommand{\theequation}{\thesection.\arabic{equation}}
\setcounter{equation}{0} \setcounter{maintheorem}{0}

This research was supported in part by the NSF of China (Grant Nos.
11801496 and 11926352), the Fok Ying-Tung Education Foundation
(China), and Hubei Key Laboratory of Applied Mathematics (Hubei
University). After we put the first version of our manuscript on
arXiv $12^{\mathrm{th}}$ March 2024, Prof. Kui Wang at Soochow
University sent us a digital copy of the reference \cite{GW} and
told us that for the Euclidean $n$-space $\mathbb{R}^n$ with
Gaussian density (i.e. $e^{-\phi}=(2\pi)^{-n/2}e^{-t^{2}/2}$), under
an extra upper bound constraint for the Gaussian density they could
obtain the spectral isoperimetric inequality (\ref{II-1}) in Theorem
\ref{theo-1} here. We would like to thank Prof. Wang for sharing
their preprint with us and for taking an interest in our manuscript.

\section*{Conflict of interest}

The authors declare that there are no conflicts of interests
regarding the publication of this paper.

\section*{Data availability statement}

Data sharing is not applicable to this article as no new data were
created or analyzed in this study.


\begin{thebibliography}{9999}


\bibitem{AB} M. S. Ashbaugh, R. D. Benguria, \emph{Universal bounds for the low
eigenvalues of Neumann Laplacians in $N$-dimensions}, SIAM J. Math.
Anal. {\bf 24} (1993) 557--570.

\bibitem{BE} D. Bakry, M. \'{E}mery, \emph{Diffusion
hypercontractives}, S\'{e}m. Prob. XIX. Lect. Notes Math. {\bf 1123}
(1985), pp. 177--206.


\bibitem{BBC} R. D. Benguria, B. Brandolini, F. Chiacchio, \emph{A sharp estimate for Neumann eigenvalues of the
Laplace-Beltrami operator for domains in a hemisphere}, Commun.
Contemp. Math. {\bf 22} (2022), Article No. 1950018.

\bibitem{BP} L. Brasco,  A. Pratelli, \emph{Sharp stability of some spectral inequalities}, Geom. Funct. Anal. {\bf 22} (2012)
107--135.

\bibitem{GRC} Gregory R. Chambers, \emph{Proof of the log-convex density conjecture}, J.
Eur. Math. Soc. {\bf 21} (2019) 2301--2332.


\bibitem{IC} I. Chavel, \emph{Eigenvalues in Riemannian Geometry}, Academic Press, New
York (1984).

\bibitem{CH} H. Chen, \emph{The upper bound of the harmonic mean of the Neumann eigenvalues in
curved spaces}, J. Geom. Phys.  {\bf 200} (2024), Article No.
105179.

\bibitem{CM} R. F. Chen, J. Mao, \emph{Several isoperimetric inequalities of Dirichlet and Neumann eigenvalues
of the Witten-Laplacian}, (2025), to appear in J. Spectral Theory.
Also available online at arXiv:2403.08075v3


\bibitem{CP} Q. M. Cheng, Y. Peng, \emph{Estimates for eigenvalues of $\mathcal{L}$ operator on self-shrinkers},
Commun. Contemp. Math. {\bf 15} (2013),
https://doi.org/10.1142/S0219199713500119


 \bibitem{CMII} T. H. Colding, W. P. Minicozzi II, \emph{Generic mean curvature flow I;
generic singularities}, Ann. of Math. {\bf175} (2012) 755--833.


\bibitem{DMCW} Y. L. Deng, J. Mao, R. F. Chen, C. X. Wu, \emph{Spectral isoperimetric inequalities for low-order Neumann
eigenvalues of the Witten-Laplacian on manifolds with nonpositive
curvatures} (in Chinese), (2024), submitted.

\bibitem{DMWW} F. Du, J. Mao, Q. L. Wang, C. X. Wu, \emph{Eigenvalue inequalities for the buckling problem of the
drifting Laplacian on Ricci solitons}, J. Differ. Equat. {\bf 260}
(2016) 5533--5564.


\bibitem{DM1} F. Du, J. Mao, \emph{Estimates for the first eigenvalue of the drifting Laplace and
the $p$-Laplace operators on submanifolds with bounded mean
curvature in the hyperbolic space}, J. Math. Anal. Appl. {\bf 456}
(2017) 787--795.



\bibitem{GW} Y. Gao, K. Wang, \emph{Isoperimetric inequalities for Neumann eigenvalues in Gauss
space}, preprint.

\bibitem{XDL} X. D. Li, \emph{Perelman's entropy formula for the Witten Laplacian on Riemannian
manifolds via Bakry-Emery Ricci curvature}, Math. Ann. {\bf 353}
(2012) 403--437.


\bibitem{LMW} W. Lu, J. Mao, C. X. Wu, \emph{A universal bound for lower Neumann eigenvalues of the Laplacian}, Czech. Math.
J. {\bf 70} (2020) 473--482.

\bibitem{LMWZ} W. Lu, J. Mao, C. X. Wu, L. Z. Zeng, \emph{Eigenvalue estimates for the drifting Laplacian and the $p$-Laplacian on submanifolds of warped
products}, Appl. Anal. {\bf 100} (2021) 2275--2300.


\bibitem{JM1} J. Mao, \emph{The Gagliardo-Nirenberg inequalities and manifolds with non-negative weighted Ricci
curvature}, Kyushu J. Math.  {\bf 70} (2016) 29--46.

\bibitem{JM2} J. Mao, \emph{Functional inequalities and manifolds with nonnegative weighted Ricci
curvature}, Czech. Math. J. {\bf 70} (2020) 213--233.

\bibitem{MDW} J. Mao, F. Du, C. X. Wu, \emph{Eigenvalue Problems on Manifolds}, Science Press, Beijing (2017).

\bibitem{MTZ} J. Mao, R. Q. Tu, K. Zeng, \emph{Eigenvalue estimates for submanifolds in Hadamard manifolds and product manifolds $N\times\mathbb{R}$},
 Hiroshima Math. J. {\bf 50} (2020) 17--42.

\bibitem{NN} N. Nadirashvilli, \emph{Conformal maps and isoperimetric inequalities for eigenvalues of the Neumann
problem},  Proceedings of the Ashkelon Workshop on Complex Function
Theory (1996), pp. 197--201. Israel Math. Conf. Proc. 11, Bar-Ilan
Univ., Ramat Gan (1997).


\bibitem{NMCW} C. X. Nie, J. Mao, R. F. Chen, Z. Z. Wang, \emph{Ashbaugh-Benguria type isoperimetric inequalities for nonzero Neumann
eigenvalues of the Laplacian on hemispheres} (in Chinese), (2024),
submitted.

\bibitem{PP} P. Petersen, \emph{Riemannian Geometry}, 2nd edition, Graduate Texts in Mathematics, Vol. 171.
Springer, New York (2006).


\bibitem{RCBM} C. Rosales, A. Ca\~{n}ete, V. Bayle, F. Morgan,
\emph{On the isoperimetric problem in Euclidean space with density},
Calc. Var. Partial Differential Equations {\bf 31} (2008) 27--46.

\bibitem{GS} G. Szeg\H{o}, \emph{Inequalities for certain eigenvalues of a membrane of given
area}, J. Rational Mech. Anal. {\bf 3} (1954) 343--356.

\bibitem{WK} K. Wang, \emph{An upper bound for the second Neumann eigenvalue on Riemannian
manifolds}, Geom. Dedicata {\bf 201} (2019) 317--323.

\bibitem{WW} G. F. Wei, W. Wylie, \emph{Comparison geometry for the Bakry-\'{E}mery Ricci tensor}, J. Differ.
Geom. {\bf 83} (2009) 377--405.

\bibitem{HFW} H. F. Weinberger, \emph{An isoperimetric inequality for the $N$-dimensional free membrane problem}, J. Rational Mech. Anal. {\bf 5}
(1956) 633--636.


\bibitem{CYX} C. Y. Xia, \emph{A universal bound for the low eigenvalues of Neumann Laplacians on
compact domains in a Hadamard manifold}, Monatsh. Math. {\bf 128}
(1999) 165--171.

\bibitem{XW} C. Y. Xia, Q. L. Wang, \emph{On a conjecture of Ashbaugh and Benguria about lower
eigenvalues of the Neumann laplacian}, Math. Ann. {\bf 385} (2023)
863--879.


\bibitem{YWMD} Y. Zhao, C. X. Wu, J. Mao, F. Du, \emph{Eigenvalue comparisons in Steklov eigenvalue problem
and some other eigenvalue estimates}, Revista Matem\'{a}tica
Complutense {\bf 33} (2020) 389--414.


\end{thebibliography}
\end{document}